\documentclass[final]{siamltex}

\usepackage{amsfonts}
\usepackage{amssymb}
\usepackage{amsbsy}
\usepackage[mathscr]{eucal}

\usepackage{pst-grad, pst-node}
\usepackage{rotating}
\usepackage{epsfig}
\usepackage{graphicx}
\usepackage{graphics}
\usepackage{floatflt}
\usepackage{verbatim}
\usepackage{psfrag}
\usepackage{color}
\usepackage{cases}

\def\Section{\setcounter{equation}{0} \setcounter{thm}{0}\section}
\def\thesection {\arabic{section}}
\def\thesubsection {\thesection.\arabic{subsection}}
\def\thesubsubsection {\thesubsection.\arabic{subsubsection}}
\def\theequation {\thesection.\arabic{equation}}
\def\be{\begin{equation}\displaystyle}
\def\ee{\end{equation}}
\def\bel{\begin{equation} \displaystyle \begin{array}{l} }
\def\eel{\end{array} \end{equation} }
\def\bell{\begin{equation} \displaystyle \begin{array}{ll}  }
\def\eell{\end{array} \end{equation} }

\def\bea{\begin{eqnarray}}
\def\eea{\end{eqnarray} }
\def\bean{\begin{eqnarray*}}
\def\eean{\end{eqnarray*} }

\def\CC{\mathbb{C}}

\def\RR{\mathbb{R}}

\catcode`@=11
\def\eqalign#1{\null\,\vcenter{\openup1\jot \m@th
   \ialign{\strut \hfil$\displaystyle{##}$ & $\displaystyle{{}##}$\hfil
      \crcr#1\crcr}}\,}
\catcode`@=12

\newtheorem{thm}{Theorem}  
\newtheorem{lem}[thm]{Lemma}
\newtheorem{cor}[thm]{Corollary}

\newtheorem{defi}[thm]{Definition}
\newtheorem{prop}[thm]{Proposition}
\newtheorem{rem}[thm]{Remark}

\def\debthm {\begin{thm}}
\def\finthm {\end{thm}}
\def\deblem {\begin{lem}}
\def\finlem {\end{lem}}
\def\debprop {\begin{prop}}
\def\finprop {\end{prop}}
\def\debcor {\begin{cor}}
\def\fincor {\end{cor}}
\def\debdef {\begin{defi}}
\def\findef {\end{defi}}
\def\debrem {\begin{rem}}
\def\finrem {\end{rem}}

\def\debbox{ 
   \hspace{10mm}\smallskip\newline 
   \hspace*{3mm}\begin{tabular}[b]{|r |l} \hline \\
   \rm\makebox[2mm]{}\begin{minipage}{\noulong}}
\def\finbox{\end{minipage}\\ \\ \hline \end{tabular}}

\def\ds{\displaystyle}

\def\eps{\varepsilon}

\def\bar#1{{\overline #1}}

\def\calO{{\cal O}}

\def\i{{i}} 

\newcommand{\eqref}[1]{(\ref{#1})}

\catcode`@=11
\def\supess{\mathop{\operator@font ess\,sup}}
\catcode`@=12

\title{WKB-method for the 1D Schr\"odinger equation in the semi-classical limit: enhanced phase treatment}

\author{Anton Arnold\thanks{Institut f\"ur Analysis und Scientific Computing, Technische Universit\"at Wien, Wiedner Hauptstr. 8-10, A-1040 Wien, Austria ({\tt anton.arnold@tuwien.ac.at}).}
\and Christian Klein\thanks{Institut de Math\'ematiques de Bourgogne, 
Universit\'e de Bourgogne-Franche-Comt\'e, 9 avenue Alain Savary, France.  ({\tt Christian.Klein@u-bourgogne.fr}).}
\and Bernhard Ujvari\thanks{Institut f\"ur Analysis und Scientific Computing, Technische Universit\"at Wien, Wiedner Hauptstr. 8-10, A-1040 Wien, Austria ({\tt e0326211@student.tuwien.ac.at}).}
}

\begin{document}

\maketitle

\begin{abstract}
This paper is concerned with the efficient numerical computation of solutions 
to the 1D stationary Schr\"odinger equation in the semiclassical 
limit in the highly oscillatory regime. A previous approach to this problem 
based on explicitly incorporating the leading terms of the WKB 
approximation is enhanced in two ways: first a refined error analysis 
for the method is presented for a not explicitly known WKB phase, and 
secondly the phase and its derivatives will be computed with spectral 
methods. The efficiency of the approach is illustrated for several 
examples. 
\end{abstract}

\begin{keywords}
 Uniformly accurate scheme, Schr\"odinger equation, highly oscillating wave functions, higher order WKB-approximation, asymptotically correct finite difference scheme, spectral methods
\end{keywords}

\begin{AMS} 35Q40, 81Q20, 65M70, 65L11
 
\end{AMS}

\pagestyle{myheadings}
\thispagestyle{plain}
\markboth{A. ARNOLD, C. KLEIN AND B. UJVARI}{WKB-method for the 1D Schr\"odinger equation
}

\Section{Introduction} \label{SEC1}
This paper is concerned with the  numerical solution of highly oscillating differential equations of the type
\be \label{EQ_ref}
\eps^2 \varphi''(x) + a(x) \varphi(x) =0\,,\quad x\in\RR\,,
\ee
where $0<\eps \ll 1$ is a very small parameter and $a(x)\ge a_0>0$ a sufficiently smooth function. 
Such models play an important role in electromagnetic and acoustic scattering (1D Maxwell and Helmholtz equations in the high frequency regime), as well as wave propagation problems in quantum and plasma
physics. 
More concretely, \eqref{EQ_ref} has been used for the simulation of electron transport in nano-scale semiconductor devices (like 1D quantum models of resonant tunneling diodes \cite{SHMS98, MJK13} or for the longitudinal mode(s) in 2D and 3D quantum waveguides \cite{lent}). In such applications $\eps:=\hbar/\sqrt{2m}$ is proportional to the (reduced) Planck constant $\hbar$, $a(x):=E-V(x)$, with some prescribed, real valued, electrostatic potential $V$, and $E\in\RR$ is the injection energy of the electrons (with effective mass $m$) into the device from the metallic leads on both sides. Since $E$ may take arbitrarily large values, \eqref{EQ_ref} appears as a Schr\"odinger equation in the semi-classical regime, i.e.\ for $\eps\to0$.

For $\eps\ll1$ the wave length $\lambda={2\pi\eps
\over \sqrt{a(x)}}$ is very
small, such that the solution $\varphi$ becomes highly oscillating. 
In classical ODE--schemes (like in \cite{Ba1, Ba2}) such a situation
requires a very  fine mesh in order to accurately resolve 
the oscillations, see Fig. \ref{onde}.
Hence, standard numerical methods would be very costly and inefficient here. 

A possible remedy 
is to use analytic a-priori information on the solution to simplify its numerical treatment. 
Within this spirit, various numerical strategies for highly oscillatory problems from quantum mechanics or of Hamiltonian structure have been developed in recent years. One group of approaches is based on 
\emph{adiabatic integrators} (see \cite{JL03,Ja04,lub}; \S XIV of \cite{HLW}), which is closely related to using the zeroth order WKB-approximation (cf.\ \cite{LL85}) $\varphi(x)\approx C \exp\big(\pm \frac{i}{e}\int_0^x \sqrt{a(\tau)}\,d\tau\big)$ to eliminate the dominant oscillations; the resulting smoother problem is then solved numerically. A localized variant of this transformation is also the background for the \emph{modified Magnus method} developed in \S5 of \cite{Is02}. 
But, according to the detailed numerical comparison in \S5.1 of \cite{lub}, their \emph{adiabatic Magnus method} is actually even more accurate than both the standard \cite{IMKNZ00, MN08} and modified Magnus methods \cite{Is02}, particularly for very small $\eps$.
\emph{Modulated Fourier expansions} are closely related techniques that apply to the case $a=const$, but allowing for nonlinearities in the ODE (see e.g.\ \S XIII of \cite{HLW}; \cite{CHL03}).

Other numerical approaches for \eqref{EQ_ref} include a macroscopic reformulation \cite{DGM07} and, more recently, WKB-approximations (see \cite{Cla,ABN11,HLH16}). The last three papers are concerned only with the oscillatory case (i.e.\ $a(x)>0$). The evanescent case (with $a(x)<0$) is equally important in applications (e.g. for quantum tunneling), but needs a different numerical approach, see \cite{AN16} for a FEM with WKB-type basis functions. For the inclusion of a turning point we refer to \cite{AD18}.

WKB-based methods rest upon using asymptotically correct solution formulas for \eqref{EQ_ref} in the semi-classical limit $\eps\to0$, in order to simplify the numerical problem. 
With the WKB-ansatz 
\begin{equation}\label{WKB}
  \varphi(x) \sim \exp\left(\frac1\eps \sum_{p=0}^\infty \eps^p\phi_p(x)\right)
\end{equation}
and a comparison of coefficients one obtains the well-known WKB-approximation in arbitrary order of accuracy w.r.t.\ powers of $\eps$ (cf.\ \cite{LL85}). 
This a-priori knowledge on the (complex valued) solution $\varphi(x)$ 
allows to separate its scales: The microscale behavior, involving the high oscillations of $\varphi$ (or its argument, $\arg \varphi$) are analytically rather accurately known. Hence, it can be eliminated from the problem to a large extent. But the macroscale behavior, corresponding to the smooth variations of $|\varphi(x)|$ has to be obtained numerically. W.r.t.\ this scale separation, the numerical WKB-method has a conceptual similarity to the \emph{heterogeneous multiscale method} (see \cite{EELRVE07} for a review). In the WKB-method, however, it is not necessary to solve the microscale problem numerically -- its contribution can be obtained analytically.

In this paper we shall extend and refine the asymptotically correct (as $\eps\to0$) WKB-scheme from  \cite{ABN11}. It is based on the second order WKB-approximation for \eqref{EQ_ref} which reads
\be\label{2orderWKB}
  \varphi(x)\approx C \,\frac{\exp\left(\pm\frac{\i}{\eps}\int_0^x\big[\sqrt{a(\tau)}-\eps^2\beta(\tau)\big]\,d\tau  \right)}{\sqrt[4]{a(x)}}\,,
\ee
with 
\be\label{beta-def}
  \beta :=-\frac{1}{2a^{1/4}}(a^{-1/4})''\,.
\ee
Since \cite{lub} is also based on a WKB-approximation (of zero order), its strategy is closely related to the procedure in \cite{ABN11}. But the latter paper yields a refinement to higher $\eps$--order. This improvement is illustrated by the numerical comparison in \S7.2 in \cite{Geier}.

The hybrid method from \cite{ABN11} consists of two steps: first an analytic preprocessing of \eqref{EQ_ref} to transform it into a smoother, i.e.\ less oscillatory problem. Then, the numerical solution of the transformed ODE problem can be carried out on a coarse grid with high accuracy. It is based on a truncated Picard iteration. This explicit representation of an approximate solution involves oscillatory integrals with two small parameters, $\eps$ and the step size $h$. Hence, the key issue in the numerical step is to construct approximations of these integrals that are $\eps$-uniform as well as accurate enough w.r.t.\ $h$, since the latter determines the local discretization error for the ODE scheme. In \cite{ABN11} both a first and second order method (w.r.t.\ $h$) were derived.

\begin{figure}[htbp]
\begin{center}
\includegraphics[width=6cm]{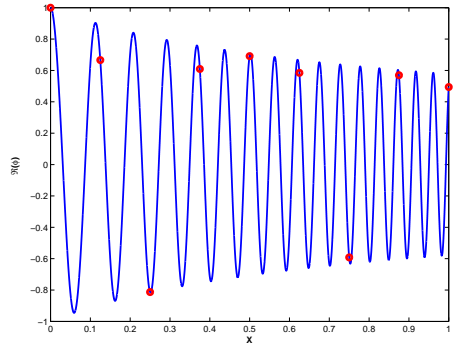}
\end{center}
\caption{\label{onde} {\footnotesize In standard numerical methods
highly oscillating solutions require a very fine mesh to resolve the oscillations.
However, with the analytic pre-processing of the WKB-marching method from \cite{ABN11} an accurate solution can be obtained on a coarse grid
(dots). Plotted is the solution $\Re[\varphi(x)]$ of \eqref{EQ_ref} with $\eps=0.01,\,h=0.125,$ and $a(x)=(x+\frac12)^2$.}}
\end{figure}

In the analytic step, the second order differential equation (\ref{EQ_ref}) is first transformed to a system of first order differential equations using the still oscillatory variable 
\be\label{U-trafo}
U(x)=\left(
\begin{array}{c}
 u_1\\[2mm]
\ds u_2
\end{array}
\right)
:=\left(
\begin{array}{c}
 a^{1/4} \varphi(x)\\[2mm]
\ds \frac{\varepsilon (a^{1/4}\varphi)'(x)}{\sqrt{a(x)}}
\end{array}
\right)\,.
\ee
In the essential second step, the dominant oscillations are transformed out of $U(x)$, yielding a first order ODE for a new, smooth variable $Z(x)$. This transformation is made precise in \S\ref{SEC21}, and the numerical solution of the ODE for $Z(x)$ is presented in \S\ref{SEC22}.

The numerical analysis of the WKB-method in \cite[Th.\ 3.1]{ABN11} led to the following error estimates for the first and, resp., second order methods:
\be \label{error_U}
||U(x_n)-U_n||_{} \le C {h^{\gamma}
\over \eps} +C \eps^2  \min(\eps,h)\,,\quad 1\le n\le N\,,
\ee
and
\be \label{error_U_2ORD}
||U(x_n)-U_n||_{} \le C {h^{\gamma}
\over \eps} +C \eps^3 h^2\,, \; 1\le n\le N\,,
\ee 
where $U_n$ is a numerical approximation for the solution $U(x)$ at the grid point $x_n$. Here and in the sequel, $C$ denotes generic, but not necessarily equal constants that are independent of grid index $n$, $h$, and $\eps$.
Moreover, $\gamma >0$ is the order of the chosen numerical integration
method for computing the phase integral 
\be \label{PH}
\phi(x):= \int_0^x \left(\sqrt{a(\tau)} -\eps^2 \beta(\tau) \right) d\tau \,,
\ee
which is a smooth function of both $x$ and $\eps$.
Here, 
$\|.\|$ denotes any vector norm in $\CC^2$.

The second terms in \eqref{error_U} and \eqref{error_U_2ORD} are the errors due to the WKB method, and they decay like $\calO(\eps^3)$, even when the step size $h$ is kept constant. By contrast, the first term is critical in the semi-classical limit as it grows for $\eps\to0$. This error is due to using an approximate phase for the analytic transformation step (from $U$ to $Z$ and back). 
To limit the size of the first term one has two options: One can either choose an $\eps$-dependent step size (like $h=\calO(\sqrt\eps)$ when using e.g.\ the Simpson rule with $\gamma=4$ for the phase computation, cf.\ \cite{lub,ABN11}). Note that this restriction is much weaker than having to resolve each oscillation (by using $h=\calO(\eps)$). Moreover it would then render the numerical scheme still second order in $h$, uniformly for $\eps$ satisfying $h=\calO(\sqrt\eps)$, see \eqref{error_U_2ORD}. 
Alternatively one can use a highly accurate method to compute the phase \eqref{PH} -- like a spectral method, and this will be our approach here.

The goal of this paper is twofold: In \cite{ABN11}, numerical errors in the phase were only considered for the backward transformation from $Z_n$, the numerical approximation of $Z$, to $U_n$. Hence we shall first complete that error analysis by taking into account that also the first (analytic) transformation (from \eqref{EQ_ref} to the ODE for $Z$) is typically affected by an inaccurate phase. Secondly, we shall combine the WKB-method with a spectral method for computing the phase function $\phi(x)$, yielding spectral accuracy. With little effort, this will reduce the first term in the error estimates \eqref{error_U} and \eqref{error_U_2ORD} to the order: $eps*\mathcal O(\eps)$, where $eps$ denotes the machine precision. This makes these error components irrelevant for most practical computations.

This paper is organized as follows: In \S\ref{SEC2} we briefly review the WKB-method from \cite{ABN11}, and in \S\ref{SEC3} we extend the error analysis to include the phase error in $\phi$. In \S\ref{SEC4} we review the Chebychev collocation method along with the Clenshaw-Curtis algorithm to compute the phase integral. In \S\ref{SEC5} we illustrate the efficiency of the combined WKB-spectral method on some numerical examples, and we conclude in \S\ref{SEC6}.

\Section{Review of the WKB-method for the stationary Schr\"odinger equation} \label{SEC2}

In this section we briefly review the WKB-based numerical method from \cite{ABN11} for solving the following scalar, highly oscillating initial value problem (IVP):
\begin{equation} \label{SchIVP}
\left\{
\begin{array}{l}
\displaystyle \varepsilon^2 \varphi''(x) + a(x) \varphi(x)= 0 \,, \quad x \in (0,1)\,,\\[2mm]
\ds \varphi(0)=\varphi_0\,,\\[2mm]
\ds \varepsilon \varphi'(0)=\varphi_1 \,,
\end{array}
\right.
\end{equation}
with possibly complex valued initial conditions. For the rest of the paper we make the following assumptions on the coefficient function $a(x)$:\\

{\bf Hypothesis A} {\it Let $a \in C^{\infty}[0,1]$ be a fixed smooth (real valued) function, satisfying  $a(x) \ge a_0 >0$ in $[0,1]$, which means that we are in the oscillatory regime.\\ 
Besides, let $0 < \eps<\eps_0$ be an arbitrary real number with 
$$
  0<\eps_0<\eps_1 := \min\big\{ 1,\,\min_{x\in[0,1]} [a(x)^{1/4}\beta_+(x)^{-1/2}] \big\},
$$
where $\beta_+(x):=\max(0,\beta(x))$.
}
\medskip

This hypothesis implies that there are positive constants $C_0$, $C_1$ such that
$$
  0<C_0\le\phi'(x)=\sqrt{a(x)}-\eps^2\beta(x)\le C_1\,,\quad x\in[0,1]\,,
$$
and that $\phi''\in L^\infty(0,1)$.

Clearly, $a$ could only be piecewise $C^\infty$. Then the ODE problem could just be restarted at an interface point of reduced regularity.\\

The WKB-method consists of two steps: first an analytic transformation of \eqref{SchIVP} into a smoother problem, and then the numerical discretization of the latter.

\subsection{Reformulation of the continuous problem} \label{SEC21}

This transformation involves three steps: Using the substitution \eqref{U-trafo},
the ODE (\ref{SchIVP}) is first transformed to a system of first order differential equations: 
\begin{equation} \label{EQU}
\left\{
\begin{array}{l}
\ds U'(x)=\left[ \frac{1}{\eps} A_0(x)+\eps A_1(x)\right] U(x)\,,\quad 0<x<1\,,\\[3mm]
\ds U(0)=U_I\,,
\end{array}
\right.
\end{equation}
with the two matrices
$$
A_0(x)=
\sqrt{a(x)}
\left(
\begin{array}{cc}
0&1\\
-1&0
\end{array}
\right)
\,; \quad A_1(x)=
\left(
\begin{array}{cc}
0&0\\
2 \beta(x)&0
\end{array}
\right)\, .
$$
Next the dominant part (w.r.t.\ $\eps$) of the resulting system matrix in \eqref{EQU}, i.e.\ $\frac1\eps A_0$, 
is diagonalized by the following change of variable:
$$Y(x) := P U(x)\,,$$
with the unitary matrix 
$$P := {1\over \sqrt{2}}
\left(
\begin{array}{cc}
i&1\\
1&i
\end{array}
\right)\quad; \quad P^{-1} = {1\over \sqrt{2}}
\left(
\begin{array}{cc}
-i&1\\
1&-i
\end{array}
\right).
$$
The final transformation step eliminates the leading oscillations 
by using the diagonal matrix
$$\Phi (x) := 
\left(
\begin{array}{cc}
\ds \phi &0\\
0&-\phi \end{array}
\right)\,,
$$
with the $\eps$-dependent, (real valued) phase function $\phi(x)$.
The change of unknown
$$Z(x) =\left(
\begin{array}{c}
 z_1\\[2mm]
\ds z_2
\end{array}
\right) 
:= e^{-{i\over \eps} \Phi(x)} Y(x)\,,$$ 
finally leads to the system
\begin{equation} \label{EQZ}
\left\{
\begin{array}{l}
\ds {dZ\over dx} = \eps B Z\,,\quad 0<x<1\,,\\[3mm]
\ds Z(0)=Z_I=Y_I=P\,U_I\,.
\end{array}
\right.
\end{equation}
Here, the matrix 
\be\label{B-def}
B(x) = 
\beta(x) \left(
\begin{array}{cc}
0&e^{-\frac{2i}{ \eps} \phi(x)}\\
e^{\frac{2i}{ \eps} \phi(x)}&0
\end{array}
\right)
\ee
is off-diagonal, $\eps$-dependent -- in fact highly oscillatory, but bounded independently of $\eps$.

The principal idea of the WKB-method for $\eps\ll1$ is as follows: Instead of solving the highly oscillatory problem (\ref{SchIVP}) on a fine mesh, one solves numerically the smooth problem 
(\ref{EQZ}) on a coarse mesh, possibly with $h>\eps$. Then the original solution is recovered by
\be \label{OI}
U(x)=P^{-1} e^{{i\over \eps} \Phi(x)} Z(x)\,.
\ee

As in \cite{ABN11}, the above transformation assumes that the phase $\phi(x)$ is known exactly on the considered interval $[0,1]$. In special cases (like piecewise linear coefficient functions $a(x)$) this is indeed possible (since $\sqrt a$ and $\beta$ are then explicitly integrable). Then, the crucial first error terms in \eqref{error_U}, \eqref{error_U_2ORD} would be absent. But in general, $\phi$ has to be obtained by a numerical quadrature, yielding an approximation $\tilde\phi(x)$. This approximate phase will be used only in the final numerical method
and in our error analysis, which generalizes the analysis from \cite{ABN11}. 

\subsection{Numerical discretization of the transformed problem} \label{SEC22}
In this section we review the discretizations of \eqref{EQZ}, as developed in \cite{ABN11}. 
Let $0=x_1 < \cdots <x_n < \cdots <x_N=1$ be a discretization of the
interval $(0,1)$ and $h:= \max_{n=1,\cdots, N-1} |x_{n+1}-x_n |$. 
The marching method from \cite{ABN11} is based on two steps, firstly a truncated Picard iteration for \eqref{EQZ} with $\bar p=1,\,2$ (corresponding to first and second order schemes, respectively):
\be \label{Decomp_Z_bis}
Z(x_{n+1}) = Z(x_{n}) + \sum_{p=1}^{\bar p} \eps^p M_p(x_{n+1};x_{n}) Z(x_{n})\,,
\ee
with the matrices
\be \label{EXP_M}
  M_1(x_{n+1};x_{n}) = \int_{x_{n}}^{x_{n+1}} B(y) \, dy\,,\quad
  M_2(x_{n+1};x_{n}) = \int_{x_{n}}^{x_{n+1}} \int_{x_{n}}^{y_1}B(y_1)B(y_2) \, dy_2dy_1\,, ...
\ee

In the second step numerical approximations for $M_1$ and $M_2$ have to be found.
Since the matrix $B$ is highly oscillatory w.r.t.\ $\eps$, the key issue is to find an $\eps$--uniform discretization of the resulting oscillatory integrals. 
Since $B\in\CC^{2\times2}$ is an off-diagonal matrix, the matrices $M_p$ are alternatingly off-diagonal (Hermitian) and diagonal (with complex conjugate entries).
Their entries are (iterated) oscillatory integrals (without stationary points).
We remark that there is a vast literature on the numerical treatment of oscillatory integrals (see \cite{INO06, Ol06}, e.g.). But there, the standard setting is to treat an integral on a fixed interval, where the integrand involves a small parameter (corresponding to $\eps$ in our case).
But our application actually involves two small parameters, $\eps$ and $h$, since \eqref{Decomp_Z_bis}, \eqref{EXP_M} constitute a stepping method for the ODE \eqref{SchIVP}. Hence, the numerical errors for the oscillatory integrals \eqref{EXP_M} must also be small in $h$, in fact of the order $\mathcal O(h^2)$ and $\mathcal O(h^3)$ (for a first and second order scheme, resp.). Due to this additional requirement, standard approaches like the \emph{asymptotic method} \cite{INO06} do not apply here (directly). In \S2.2 of \cite{ABN11} a variant of this latter method was presented that allows to ``trade in'' $h$--powers in the error for $\eps$--powers. Next we summarize the resulting method, and refer the reader to \cite{ABN11} for a detailed derivation.

To present the discrete analogs of \eqref{Decomp_Z_bis}-\eqref{EXP_M} we need some notation. We define the following functions
$$
H_1(x) := e^{ix} -1\,,\qquad H_2(x) := e^{ix} -1-ix\,, 
$$
\be\label{def-betak}
\quad\beta_0(y) := \frac{\beta}{2 \phi'} (y)= \frac{\beta}{2 (\sqrt a -\eps^2\beta)} (y), 
\quad \beta_{k}(y) := \frac{1}{2 \phi'(y)} {d\over dy}\left( {\beta_{k-1} }\right)(y),\quad k=1,2,3\,,
\ee
and the phase increments 
$$
S_{n}\! := \phi(x_{n+1}) - \phi(x_{n})
=\! \int_{x_n}^{x_{n+1}} \!\!\! \left(\sqrt{a(\tau)}-\eps^2 \beta(\tau)\right) d\tau.
$$

Now we recall from \cite{ABN11} the two marching methods for the vector $Z$. 

\noindent
{\bf First order scheme.\,} 
Let $Z_1:=Z_I$ be the initial condition and let $n=1,\cdots,N-1$. Then we define
\be \label{1ORD}
Z_{n+1} := (I + A_{n}^1) Z_{n}\,,
\ee
with the $2 \times 2$ matrix
\be \label{1ORD_A}
\!\!\!\!\!\!\!\!\begin{array}{lll}
\ds &&\!\!\!\!A_{n}^1 := \ds \eps^3 \beta_1(x_{n+1})\left(\!
\begin{array}{cc}
0& e^{-{2i \over \eps} \phi(x_{n})} H_1(-{2 \over \eps} S_{n})\\
e^{{2i \over \eps} \phi(x_{n})} H_1({2 \over \eps} S_{n})&0
\end{array}
\! \right)\\[6mm]
&&\ds\! - i \eps^2\! \left(\!\!\!
\begin{array}{cc}
0&\hspace{-7mm} \beta_0(x_{n})e^{-{2i \over \eps} \phi(x_{n})} -\beta_0(x_{n+1})e^{-{2i \over \eps} \phi(x_{n+1})}\\
\beta_0(x_{n+1}) e^{{2i \over \eps} \phi(x_{n+1})} -\beta_0(x_{n})e^{{2i \over \eps} \phi(x_{n})}
&\hspace{-7mm}0
\end{array}
\!\!\!\right)\!.
\end{array}
\ee

\noindent
{\bf Second order scheme.\,} \qquad Let $Z_1:=Z_I$ be the initial condition and let $n=1,\dots,N-1$. Then we define
\be \label{2ORD}
Z_{n+1} := (I + A_{n}^2+ A_{n}^3) Z_{n}\,,
\ee
with the matrices
\be \label{2ORD_A1}
\begin{array}{lll}
\ds  A_{n}^2:=& \\[3mm]
&&\hspace{-15mm} - i \eps^2 \!\!\left(\!\!
\begin{array}{cc}
0&\hspace{-8mm} \beta_0(x_{n})e^{-{2i \over \eps} \phi(x_{n})} -\beta_0(x_{n+1})e^{-{2i \over \eps} \phi(x_{n+1})}\\[3mm]
\beta_0(x_{n+1}) e^{{2i \over \eps} \phi(x_{n+1})} -\beta_0(x_{n})e^{{2i \over \eps} \phi(x_{n})}
&\hspace{-8mm} 0
\end{array}
\!\!\!\right)\\[6mm]
&&\hspace{-15mm} \ds + \eps^3 \!\!\left(\!\!
\begin{array}{cc}
0&\hspace{-8mm} \beta_1(x_{n+1})e^{-{2i \over \eps} \phi(x_{n+1})} -\beta_1(x_{n})e^{-{2i \over \eps} \phi(x_{n})}\\[3mm]
\beta_1(x_{n+1}) e^{{2i \over \eps} \phi(x_{n+1})} -\beta_1(x_{n})e^{{2i \over \eps} \phi(x_{n})}
&\hspace{-8mm} 0
\end{array}
\!\!\right)\\[6mm]
&&\hspace{-15mm} \ds + i \eps^4 \beta_2(x_{n+1})  \left(
\begin{array}{cc}
0& -e^{-{2i \over \eps} \phi(x_{n})} H_1(-{2 \over \eps} S_{n})\\[3mm]
e^{{2i \over \eps} \phi(x_{n})} H_1({2 \over \eps} S_{n})&0
\end{array}
\! \right)\\[6mm]
&&\hspace{-15mm} \ds -  \eps^5 \beta_3(x_{n+1})  \left(
\begin{array}{cc}
0& e^{-{2i \over \eps} \phi(x_{n})} H_2(-{2 \over \eps} S_{n})\\[3mm]
e^{{2i \over \eps} \phi(x_{n})} H_2({2 \over \eps} S_{n})&0
\end{array}
\! \right)\ ,
\end{array}
\ee
\be \label{2ORD_A2}
\begin{array}{lll}
\ds  A_{n}^3&:=& \ds - i \eps^3 (x_{n+1} -x_n) { \beta(x_{n+1}) \beta_0(x_{n+1}) +\beta(x_{n}) \beta_0(x_{n}) \over 2} 
\left(
\begin{array}{cc}
\ds 1&0\\
\ds  0&\ds -1
\end{array}
\right)\\[6mm]
&& \ds - \eps^4 \beta_0(x_{n}) \beta_0(x_{n+1}) 
\left(
\begin{array}{cc}
\ds H_1(-{2\over \eps}S_{n})&0\\[5mm]
\ds 0& \ds  H_1({2\over \eps}S_{n})
\end{array}
\right)\\[7mm]
&& \ds +i  \eps^5\beta_1(x_{n+1})  [\beta_0(x_{n})-\beta_0(x_{n+1})] \left( 
\begin{array}{cc}
\ds H_2(-{2\over \eps}S_{n})
&\ds 0\\[4mm]
\ds 0
&\ds -H_2({2\over \eps}S_{n})
\end{array}
\right)\,.
\end{array}
\ee 

In order to compute now the numerical approximation of (\ref{EQU}) and thus to obtain the wave function $\varphi$ as well as its derivative $\varphi'$, we have to transform back via
\be \label{Transfo_ZU}
U_n=P^{-1}e^{{i \over \eps} \Phi (x_n)} Z_n\,, \quad  n=1,...,N\,.
\ee

In these schemes we assumed so far that the phase $\phi$, and the functions $\beta$ and $\beta_k$, $k=0,1,2,3$ (which involve up to five derivatives of $a$) are explicitly ``available''. For $\beta$ and $\beta_k$ this is feasible, as they involve only derivatives of $a$. 
But typically the phase and phase increment have to be replaced by their numerical approximates $\tilde\phi$ and $\tilde S_n:= \tilde\phi(x_{n+1}) - \tilde\phi(x_{n})$. At this point we have two options: On the one hand we could use the approximate $\tilde\phi$ along with the exact functions $a$, $\beta$, and $\beta_k$. 
On the other hand it appears more consistent to combine $\tilde\phi$ with its exact derivatives that lead to numerical approximates $\tilde a$, $\tilde \beta$, $\tilde \beta_k$. Since we shall use a spectral integration of the phase (cf.\ \S\ref{SEC4}), the exact derivatives of $\tilde\phi$ are readily available. We shall use the second option here, as this will simplify the error analysis in \S\ref{SEC3}. Moreover, the error plots of these two versions are almost indistinguishable.

With the replacements $\tilde\phi$, $\tilde S_n$, $\tilde \beta$, and $\tilde \beta_k$, the above matrices \eqref{1ORD_A}, \eqref{2ORD_A1}, and \eqref{2ORD_A2} shall be called $\tilde A_n^1$, $\tilde A_{n}^2$, and $\tilde A_n^3$.
Hence, the first and second order methods with approximate phase read
\be \label{1ORDtilde}
\tilde Z_{n+1} := (I + \tilde A_{n}^1) \tilde Z_{n}\,, \quad  n=1,...,N-1\,,
\ee
and
\be \label{2ORDtilde}
\tilde Z_{n+1} := (I + \tilde A_{n}^2+ \tilde A_{n}^3) \tilde Z_{n}\,, \quad  n=1,...,N-1\,.
\ee
For both methods, the corresponding back-transformation to the variable $U$ then reads
\be \label{Transfo_ZU_tilde}
\tilde U_n=P^{-1}e^{{i \over \eps} \tilde\Phi(x_n)} \tilde Z_n\,, \quad  n=1,...,N\,,
\ee
with $\tilde\Phi := \diag(\tilde\phi,-\tilde\phi)$. 
Note that these two WKB-schemes were introduced as a numerical approximation of the schemes \eqref{1ORD}, \eqref{2ORD} for the (transformed) Schr\"odinger equation \eqref{EQU} with the (unperturbed) coefficients $a$, $\beta$ --- just because the exact phase $\phi$ is typically not available. But, in parallel, the two schemes \eqref{1ORDtilde}, \eqref{2ORDtilde} can also be considered as WKB-schemes \emph{without a numerically perturbed phase} for the perturbed Schr\"odinger equation \eqref{EQU}, i.e.\ with the perturbed coefficients $\tilde a$, $\tilde \beta$. The second point of view will be helpful in the subsequent numerical analysis.\\

In Section \ref{SEC3} we shall give a complete error analysis of the first order scheme (\ref{1ORDtilde}), (\ref{Transfo_ZU_tilde}) and of the second order scheme (\ref{2ORDtilde}), (\ref{Transfo_ZU_tilde}). This is a completion of the error estimates given in \cite{ABN11}, since that paper used the approximate phase $\tilde\phi$ only in the back-transformation (\ref{Transfo_ZU_tilde}) but not in the matrices $A_n^1$, $A_{n}^2$, and $A_n^3$.

\Section{Error analysis including phase errors} \label{SEC3}

First we decompose the exact phase from \eqref{PH} as $\phi(x)=\phi_1(x)-\eps^2\phi_2(x)$ with
\be\label{phi-decomp}
  \phi_1(x):= \int_0^x \sqrt{a(\tau)} d\tau \,,\quad \phi_2(x):= \int_0^x \beta(\tau)  d\tau \,.
\ee
For the approximate phase $\tilde\phi=\tilde\phi_1-\eps^2\tilde\phi_2$ we shall make the following assumptions:\\

{\bf Hypothesis B} {\it Let $\tilde\phi_1, \, \tilde\phi_2\in C^\infty[0,1]$ 
with $\tilde\phi_1'(x)\ge C_2>0$.}\\

This Hypothesis implies that there are positive constants $C_3,\,C_4,\,C_5$ such that
\begin{eqnarray*}
  0<C_3\le \tilde\phi'(x)=\tilde\phi_1'(x)-\eps^2\tilde\phi_2'(x)\le C_4
  \quad && \mbox{for any }\: 0<\eps\le\tilde\eps_0\,,\\
  \tilde\phi''(x)=\tilde\phi_1''(x)-\eps^2\tilde\phi_2''(x)\le C_5
  \quad && \mbox{for any }\: 0<\eps\le\tilde\eps_0\,,\\
\end{eqnarray*}
with some $0<\tilde\eps_0\le\eps_0$.
Hence, Hypothesis {\bf B} implies that there are positive numbers $E$, $E'$, $E''$, such that 
the following error bounds on $\tilde\phi$ hold uniformly in $\eps\in(0,\tilde\eps_0]$:
\begin{eqnarray}\label{phitilde-errors}
  \|\phi-\tilde\phi\|_{L^\infty(0,1)}&\le E\,,\nonumber\\
  \|\phi'-\tilde\phi'\|_{L^\infty(0,1)}\le 
  \|\phi_1'-\tilde\phi_1'\|_{L^\infty(0,1)}+\eps^2 \|\phi_2'-\tilde\phi_2'\|_{L^\infty(0,1)}&\le E'\,,\\
  \|\phi''-\tilde\phi''\|_{L^\infty(0,1)}\le 
  \|\phi_1''-\tilde\phi_1''\|_{L^\infty(0,1)}+\eps^2 \|\phi_2''-\tilde\phi_2''\|_{L^\infty(0,1)}&\le E''\,,\nonumber
\end{eqnarray}
These error bounds will be important ingredients for the subsequent error estimates.
In accordance with \eqref{phi-decomp} we also define $\tilde a(x):=(\phi_1'(x))^2$, $\tilde\beta(x):=\phi_2'(x)$, but we do \emph{not require} that $\tilde\beta =-\frac{1}{2\tilde a^{1/4}}(\tilde a^{-1/4})''$ holds (cp. with \eqref{beta-def}). We remark that this equality was also \emph{not} used in the error analysis of \cite{ABN11}.

Note that we assume here that $\tilde\phi$ is a continuous (and smooth) function on $[0,1]$, and it is not only defined on the grid points $x_n$. In particular, this is satisfied for the spectral approximation constructed in \S\ref{SEC4} below.\\

As a first step of the error analysis we shall estimate the error between the (continuous) solution $Z(x)$ to \eqref{EQZ} and its perturbed analog $\tilde Z(x)$, which is the exact solution to
\begin{equation} \label{EQZ1}
\left\{
\begin{array}{l}
\ds {d\tilde Z\over dx} = \eps \tilde B \tilde Z\,,\quad 0<x<1\,,\\[3mm]
\ds \tilde Z(0)=Z_I=P\,U_I\,,
\end{array}
\right.
\end{equation}
with the matrix 
$$
\tilde B(x): = 
\tilde \beta(x) \left(
\begin{array}{cc}
0&e^{-\frac{2i}{ \eps} \tilde \phi(x)}\\
e^{\frac{2i}{ \eps} \tilde \phi(x)}&0
\end{array}
\right)\,.
$$

\begin{lemma}\label{Z-diff}
Let the coefficient function $a$ satisfy Hypothesis {\bf A} and let $\tilde\phi$ satisfy Hypothesis {\bf B}. Then we have
\be\label{Z-error}
  \|Z-\tilde Z\|_{L^\infty(0,1)}\le c\eps [\min(\eps,E)+\eps(E'+E'')]\,,
\ee
with some generic constant $c$ independent of $\eps\in(0,\tilde \eps_0]$.
\end{lemma}
\begin{proof}
\underline{Step 1 (bound on the solution propagator:)}
We define the \emph{propagator} pertaining to the ODE in \eqref{EQZ} as the matrix $S(x,y)\in\CC^{2\times 2}$ that satisfies $Z(x)=S(x,y)\,Z(y)$, and $\tilde S(x,y)$ is the propagator for \eqref{EQZ1}.
To estimate the growth of the solution $Z(x)\in\CC^2$ we compute
\begin{eqnarray*}
  \frac{d}{d x} \|Z\|^2 &=& 2\eps\beta(x) \bar Z^T 
  \left(
\begin{array}{cc}
0&e^{-\frac{2i}{ \eps} \phi(x)}\\
e^{\frac{2i}{ \eps} \phi(x)}&0
\end{array}
\right) Z 
= 2\eps\beta(x) \Big(e^{\frac{2i}{ \eps} \phi(x)} z_1\bar z_2 + e^{-\frac{2i}{ \eps} \phi(x)}\bar z_1z_2 \Big) \\
&\le& 2\eps\|\beta\|_\infty \|Z\|^2\,.
\end{eqnarray*}
Hence $\frac{d}{d x} \|Z\|\le \eps\|\beta\|_\infty \|Z\|$, where we used the abbreviation $\|.\|_\infty$ for $\|.\|_{L^\infty(0,1)}$. Gronwall's lemma then implies for the solution propagator of \eqref{EQZ}:
\be\label{S-est}
  \|S(x,y)\| \le e^{\eps\|\beta\|_\infty |x-y|}\,,\qquad 0\le x,\,y\le1\,,
\ee
and analogously for the solution propagator of \eqref{EQZ1}:
\be\label{tildeS-est}
  \|\tilde S(x,y)\| \le e^{\eps\|\tilde \beta\|_\infty |x-y|}\,\qquad 0\le x,\,y\le1\,.
\ee

\noindent
\underline{Step 2 (bound on $\Delta Z$:)}
We denote $\Delta Z:=Z-\tilde Z$, which satisfies
$$
  \Delta Z' = \eps B(x)\Delta Z +\eps(B(x)-\tilde B(x))\tilde Z(x)\,,\quad \Delta Z_I=0\,.
$$
The solution of this inhomogeneous equation reads
\begin{eqnarray}\label{DZ-repr}
\Delta Z &=& \eps\int_0^x S(x,y)\,[(B(y)-\tilde B(y))\tilde Z(y)]\,dy \nonumber\\
&=& i\eps^2 \int_0^x S(x,y)\, 
\left(
\begin{array}{c}
\big[\beta_0 \big(e^{-\frac{2i}{ \eps} \phi}\big)' -\tilde\beta_0 \big(e^{-\frac{2i}{ \eps} \tilde\phi}\big)'\big]\,\tilde z_2 \\
\big[-\beta_0 \big(e^{\frac{2i}{ \eps} \phi}\big)' +\tilde\beta_0 \big(e^{\frac{2i}{ \eps} \tilde\phi}\big)'\big]\,\tilde z_1
\end{array} 
\right) \,dy\nonumber \\
&=& i\eps^2 S(x,y)\, 
\left(
\begin{array}{c}
\big[\beta_0 e^{-\frac{2i}{ \eps} \phi} -\tilde\beta_0 e^{-\frac{2i}{ \eps} \tilde\phi}\big]\,\tilde z_2 \\
\big[-\beta_0 e^{\frac{2i}{ \eps} \phi} +\tilde\beta_0 e^{\frac{2i}{ \eps} \tilde\phi}\big]\,\tilde z_1
\end{array} 
\right) \Bigg|_{y=0}^{y=x} \\
&-&i\eps^2 \int_0^x S(x,y)\, 
\left(
\begin{array}{c}
(\beta_0\tilde z_2)' e^{-\frac{2i}{ \eps} \phi} -(\tilde\beta_0\tilde z_2)' e^{-\frac{2i}{ \eps} \tilde\phi} \\
-(\beta_0\tilde z_1)' e^{\frac{2i}{ \eps} \phi} +(\tilde\beta_0\tilde z_1)' e^{\frac{2i}{ \eps} \tilde\phi} 
\end{array} 
\right) \,dy \nonumber \\
&+&i\eps^3 \int_0^x S(x,y)\, \Bigg[B(y)
\left(
\begin{array}{c}
\big[\beta_0 e^{-\frac{2i}{ \eps} \phi} -\tilde\beta_0 e^{-\frac{2i}{ \eps} \tilde\phi}\big]\,\tilde z_2 \\
\big[-\beta_0 e^{\frac{2i}{ \eps} \phi} +\tilde\beta_0 e^{\frac{2i}{ \eps} \tilde\phi}\big]\,\tilde z_1
\end{array} 
\right)\Bigg] \,dy\,,\nonumber
\end{eqnarray}
where we skipped in the integrands the argument "$(y)$" for brevity.
In \eqref{DZ-repr} we used the following integration by parts formula involving the propagator $T(x,y)$ for some linear evolution equation $u'=A(x)u$:
\begin{eqnarray*}
  && \int_0^x T(x,y)\,[f'(y)g(y)]\,dy = T(x,y)\,[f(y)g(y)]\big|_{y=0}^{y=x} \\
  && - \int_0^x T(x,y)\,[f(y)g'(y)]\,dy + \int_0^x T(x,y)\,\big[A(y)\{f(y)g(y)\}\big]\,dy\,,
\end{eqnarray*}
which can be verified easily by using 
$\frac{\partial}{\partial y} T(x,y) = -T(x,y)A(y)$.
Note that the integration by parts in the oscillatory integral of \eqref{DZ-repr} will allow to recover one more $\eps$-power in the estimate of $\Delta Z$. This strategy was already used in Proposition 2.2 of \cite{ABN11}.
In the last line of \eqref{DZ-repr} we used also 
\be\label{dSdy}
  \frac{\partial}{\partial y}S(x,y)=-\eps S(x,y)B(y) \,.
\ee

On the r.h.s.\ of \eqref{DZ-repr} we have to consider two types of differences: First we estimate
\begin{eqnarray}\label{est1}
  && \Big|\beta_0(y)e^{\frac{2i}{\eps} \phi(y)} - \tilde\beta_0(y)e^{\frac{2i}{\eps} \tilde\phi(y)}\Big|\nonumber\\
  && =\Big|\beta_0(y)\big[ e^{\frac{2i}{\eps} \phi(y)}- e^{\frac{2i}{\eps} \tilde\phi(y)}\big]
  +\Big[-\frac{\beta(y)}{2}\frac{\phi'(y)-\tilde\phi'(y)}{\phi'(y)\tilde\phi'(y)}
  +\frac{\beta-\tilde\beta}{2\tilde\phi'}\Big] e^{\frac{2i}{\eps} \tilde\phi(y)} 
  \Big|\\
  &&\le 2\|\beta_0\|_\infty\min(1,\frac{E}{\eps}) +\|\beta\|_\infty\frac{E'}{2C_0C_3}+\frac{E'}{2C_3}\,,\nonumber
\end{eqnarray}
where we used for the first term in \eqref{est1} both the trivial estimate 
$|e^{\frac{2i}{\eps} \phi(y)} -e^{\frac{2i}{\eps} \tilde\phi(y)}|\le2$ and the mean value theorem for vector functions. We also used $\beta-\tilde\beta=\phi_2'-\tilde\phi_2'$, and we recall the definitions $\beta_0(y):=\beta(y)/(2\phi'(y))$, $\tilde\beta_0(y):=\tilde\beta(y)/(2\tilde\phi'(y))$.

Secondly we estimate:
\begin{eqnarray*}\label{est2}
  && \Big|\beta_0'(y) e^{\frac{ 2i}{\eps} \phi(y)} - \tilde\beta_0'(y) e^{\frac{ 2i}{\eps} \tilde\phi(y)}\Big|
  =\Big|\beta_0'\big[ e^{\frac{2i}{\eps} \phi}- e^{\frac{2i}{\eps} \tilde\phi}\big]- \nonumber\\
  && 
  -\Big[\frac{\beta'}{2}\frac{\phi'-\tilde\phi'}{\phi'\tilde\phi'}
  +\frac{\beta}{2} \big\{\frac{\phi''(\tilde\phi'-\phi')(\tilde\phi'+\phi')}{(\phi'\tilde\phi')^2} + \frac{\phi''-\tilde\phi''}{(\tilde\phi')^2}\big\}
  -\frac{\beta'-\tilde\beta'}{2\tilde\phi'} + \frac{(\beta-\tilde\beta)\tilde\phi''}{2(\tilde\phi')^2}
  \Big] e^{\frac{2i}{\eps} \tilde\phi}\Big|\\
  &&\le 2\|\beta_0'\|_\infty\min(1,\frac{E}{\eps}) +\|\beta'\|_\infty\frac{E'}{2C_0C_3}
  +\|\beta\|_\infty\frac{\|\phi''\|_\infty (C_1+C_4) E'}{2C_0^2C_3^2}
  +\|\beta\|_\infty\frac{E''}{2C_3^2}+\frac{E''}{2C_3}+\frac{C_5E'}{2C_3^2}\,.\nonumber
\end{eqnarray*}

{}From \eqref{S-est}, \eqref{tildeS-est}, and \eqref{EQZ1} we recall that the propagator $S(x,y)$, the solution $\tilde Z(x)$, and $\tilde Z'(x)$ are uniformly bounded in $x$, $y$, and $\eps$. Hence, the estimates \eqref{est1} and \eqref{est2} yield the result \eqref{Z-error} with a constant $c$ that depends only on $\|\beta\|_{\infty}$, $\|\beta_0\|_{W^{1,\infty}}$, $\|\phi''\|_\infty$, $C_0$, $C_1$, $C_3$, $C_4$, and $C_5$.
\end{proof}
\bigskip

This lemma allows to derive the main result of this section:
\begin{theorem}\label{ZU-est}
Let the coefficient function $a$ satisfy Hypothesis {\bf A} and let $\tilde\phi$ satisfy Hypothesis {\bf B}. Then the first order scheme \eqref{1ORDtilde}, \eqref{Transfo_ZU_tilde} and the second order scheme \eqref{2ORDtilde}, \eqref{Transfo_ZU_tilde} satisfy the following error estimates:
\be \label{error_Unew}
\quad\qquad ||U(x_n)-\tilde U_n||_{} \le C {E
\over \eps} +C \eps^2  \min(\eps,h)+C\eps [\min(\eps,E)+\eps(E'+E'')]\,,\; 1\le n\le N\,,
\ee
and
\be \label{error_U_2ORDnew}
\qquad ||U(x_n)-\tilde U_n||_{} \le C {E\over \eps} 
+C \eps^3 h^2+C\eps [\min(\eps,E)+\eps(E'+E'')]\,, \; 1\le n\le N\,,
\ee 
with some generic constant $C$ independent of $\eps\in(0,\tilde \eps_0]$, $n$, and $h$.
\end{theorem}

Let us compare this result with the estimates \eqref{error_U}, \eqref{error_U_2ORD} that are due to \cite[Th. 3.1]{ABN11}: The first error terms on the r.h.s.\ of  \eqref{error_Unew} and \eqref{error_U_2ORDnew} are generalizations to $h$-independent numerical integrations of the phase integral. The new (additional) third terms are due to using the perturbed phase $\tilde\phi$ in the WKB-method.\\

\begin{proof}[of Theorem \ref{ZU-est}]
First we estimate the error in the $Z$-variable, for $1\le n\le N$:
\begin{eqnarray}\label{Z-error1}
  ||Z(x_n)-\tilde Z_n|| &\le& \|Z-\tilde Z\|_\infty + \|\tilde Z(x_n)-\tilde Z_n\| \nonumber \\
  &\le& C\eps [\min(\eps,E)+\eps(E'+E'')] + 
  \left\{ \begin{array}{ll} C \eps^2  \min(\eps,h) & \mbox{for first order,} \\
  C \eps^3 h^2  & \mbox{for second order,} \end{array} \right.
\end{eqnarray}
where we used Lemma \ref{Z-diff} and the WKB-error estimate for $\tilde Z$ that is analogous to \eqref{error_U}, \eqref{error_U_2ORD} (skipping the first term, cf.\ \cite[Th. 3.1]{ABN11}).

Next we estimate the error in the $U$-variable, using \eqref{Transfo_ZU}, \eqref{Transfo_ZU_tilde}:
\begin{eqnarray*}
   ||U(x_n)-\tilde U_n|| &\le& \|P^{-1}\big( e^{{i \over \eps}\Phi(x_n)}- e^{{i \over \eps} \tilde\Phi(x_n)} \big)Z(x_n)\| + \|P^{-1} e^{{i \over \eps}\tilde\Phi(x_n)}\big(Z(x_n)-\tilde Z_n\big)\| \\
  &\le& C \min(1,{E\over \eps}) + ||Z(x_n)-\tilde Z_n||\,,
\end{eqnarray*}
where we used an estimate like for the first term of \eqref{est1}, the $\eps$-uniform boundedness of $Z(x)$ (see \eqref{S-est}), and the unitarity of the matrices $P^{-1}$, $e^{{i \over \eps}\tilde\Phi(x_n)}$. Combining the two estimates yields the result.
\end{proof}

\Section{Spectral integration of the phase}
\label{SEC4}
The estimation of the numerical errors (\ref{error_U}) and 
(\ref{error_U_2ORD}) in the computation of a 
solution to the Schr\"odinger 
equation in the semi-classical limit via the approach of \cite{ABN11} 
indicates that the problematic term for small $\eps$ is the first on 
the right hand sides of these expressions. It arises from the numerical 
computation of the phase (\ref{PH}) and is not present if the 
latter can be computed exactly. In cases where this is not possible, 
a high order method is recommended to reduce as much as possible 
the role of the term proportional to $1/\eps$. We use here spectral methods 
which are known to approximate analytic functions with spectral 
accuracy, i.e., an error 
decreasing exponentially with the number of modes. The numerical 
error for $C^{\infty}$ functions in accordance with Hypothesis 
\textbf{A} is known to decrease faster than any power of $h$, which 
means in practice an exponential decrease, too, see e.g.~\cite{Tr00}. 
Concretely we 
apply a Chebychev collocation method and use the Clenshaw-Curtis 
\cite{CC60}
algorithm for the integral in (\ref{PH}). For points in between 
collocation points of the Clenshaw-Curtis algorithm, we use barycentric interpolation, see \cite{BT04}.

The basic idea of spectral methods is to approximate a function $f$ on the 
interval $[a,b]$ via functions which are globally smooth on the 
considered interval. We will use here Chebychev polynomials since 
Chebychev series are related to Fourier series for which efficient 
numerical algorithms exist. Since any finite interval $x\in[a,b]$ can 
be mapped via $x = b(1+l)/2+a(1-l)/2$ to the interval $l\in[-1,1]$, we 
present all algorithms for the latter interval. We 
approximate $f(l)$ via
\begin{equation}
    f(l) \approx \sum_{n=0}^{N}a_{n}T_{n}(l),\quad l\in[-1,1]
    \label{approx},
\end{equation}
where the Chebychev polynomials $T_{n}(l)$ are defined as 
\begin{equation}
    T_{n}(l)= \cos[n\arccos(l)], \quad n = 0,1,\ldots
    \label{cheb}
\end{equation}

The idea of a 
collocation method is to introduce  collocation points $l_{j}$, 
$j=0,\ldots,N$  on $[-1,1]$ and to impose in (\ref{approx}) 
equality at the collocation points,
\begin{equation}
    f(l_{j}) = \sum_{n=0}^{N}a_{n}T_{n}(l_{j}),\quad j = 0,\ldots,N
    \label{coll}.
\end{equation}
These are $N+1$ equations to determine the spectral coefficients 
$a_{n}$, $n=0,\ldots,N$. Choosing the $l_{j}$ as the Chebychev 
collocation points $l_{j}=\cos(j\pi/N)$, $j=0,\ldots,N$ (note in 
particular that these points avoid the Runge phenomenon in 
interpolation on equidistant points and allow a uniform accuracy in 
the interpolation, see e.g., the discussion in Chap. 5 of 
\cite{Tr00}), the equations 
(\ref{coll}) take the form
\begin{equation}
    f(l_{j}) = \sum_{n=0}^{N}a_{n}\cos\left(\frac{nj\pi}{N}\right),\quad j = 0,\ldots,N
    \label{dct}.
\end{equation}
Thus the spectral coefficients are given by the discrete cosine 
transformation (DCT) of the function $f$ at the collocation points. 
Since the DCT is related to the discrete Fourier transform, it can be 
computed with the fast Fourier transform algorithm after some 
preprocessing, see for instance Chap. 8 of \cite{Tr00}. Thus one advantage of a
Chebychev collocation method is that a fast algorithm to compute the 
spectral coefficients exists. 

To integrate a function approximated by the Chebychev sum 
(\ref{approx}), a very efficient algorithm exists due to Clenshaw and 
Curtis \cite{CC60}. The basis of the algorithm is the well known 
identity for Chebychev polynomials (simply a consequence of the 
addition theorems for trigonometric functions)
\begin{equation}
      \frac{T_{n+1}'(l)}{n+1} - \frac{T_{n-1}'(l)}{n-1} = 
      2T_{n}(l),\quad n>1.
    \label{chebint}
\end{equation}
The antiderivative of a function $f$ approximated as a Chebychev sum 
(\ref{approx}) can itself be approximated by such a sum,
\begin{eqnarray}
    &&
    \int_{-1}^{l}f(l')dl' \approx \sum_{n=0}^{N}\int_{-1}^{l}a_{n}T_{n}\nonumber\\
    &&=
    a_{0}(l+1)+\sum_{n=1}^{N}a_{n}\left(
    \frac{T_{n+1}(l)-(-1)^{n}}{n+1} - \frac{T_{n-1}(l)-(-1)^{n}}{n-1}\right)
    \approx\sum_{n=0}^{N}b_{n}T_{n}(l)
    \label{bn},
\end{eqnarray}
where the $b_{n}$ follow from the $a_{n}$ via (\ref{chebint}),
\begin{eqnarray}
    b_{n}&=&\frac{1}{2(n-1)}(a_{n-1}-a_{n+1}),\quad n = 
    2,3,\ldots,N-1, \nonumber\\
    b_{1}&=&a_{0}-\frac{1}{2}a_{2}, \nonumber\\
    b_{N}&=&\frac{a_{N-1}}{2(N-1)}, \nonumber\\
    b_{0}&=&-\sum_{n=1}^{N}(-1)^{n}b_{n}
    \label{bn2}.
\end{eqnarray}
In 
\cite{Ch68} the numerical error of the Clenshaw-Curtis algorithm was 
discussed showing that it is a spectral method. 
Identity (\ref{chebint}) can obviously also be used to approximate derivatives 
of functions 
in coefficient space.  Alternatively and with the same 
numerical accuracy, one can use the differentiation matrices of Chap. 
6 of \cite{Tr00} following from Lagrangian interpolation on Chebychev 
collocation points. We use these matrices to 
compute the derivatives appearing in the definition of the 
$\beta_{k}$, $k=0,1,2,3$ (\ref{def-betak}). Thus these derivatives 
are also computed with spectral accuracy. 
For $L^\infty$--error bounds of the Chebychev spectral approximation (and its derivatives) we refer to Theorem 5 and 6 in \cite{Tr00}. We recall that such estimates were assumed for the error analysis in \S\ref{SEC3}.

As an 
example we consider the function $f(x)=\exp(-x^{2}/2)$ appearing in 
the examples in the following section for 
$x\in[0,1]$. The difference between the Clenshaw-Curtis 
approximation of $\int_{0}^{1}f(x)dx$ and the exact value 
$\sqrt{\pi/2}\,\mbox{erf}(1/\sqrt{2})$ (the well known \emph{error function} computed in 
Matlab to machine precision via $\mbox{erf}(x)$) in 
dependence of the number $N$ of collocation points is shown on the 
left of Fig.~\ref{ccfig}. It can be seen in the semilogarithmic plot 
that the numerical error decreases exponentially with the number $N$ 
of collocation points up to $N=14$ where the numerical error 
reaches the saturation level (we work here in double precision, thus the accuracy is 
limited in practice to the order of $10^{-16}$ because of rounding 
errors). 
\begin{figure}[htbp]
\begin{center}
\includegraphics[width=0.49\textwidth]{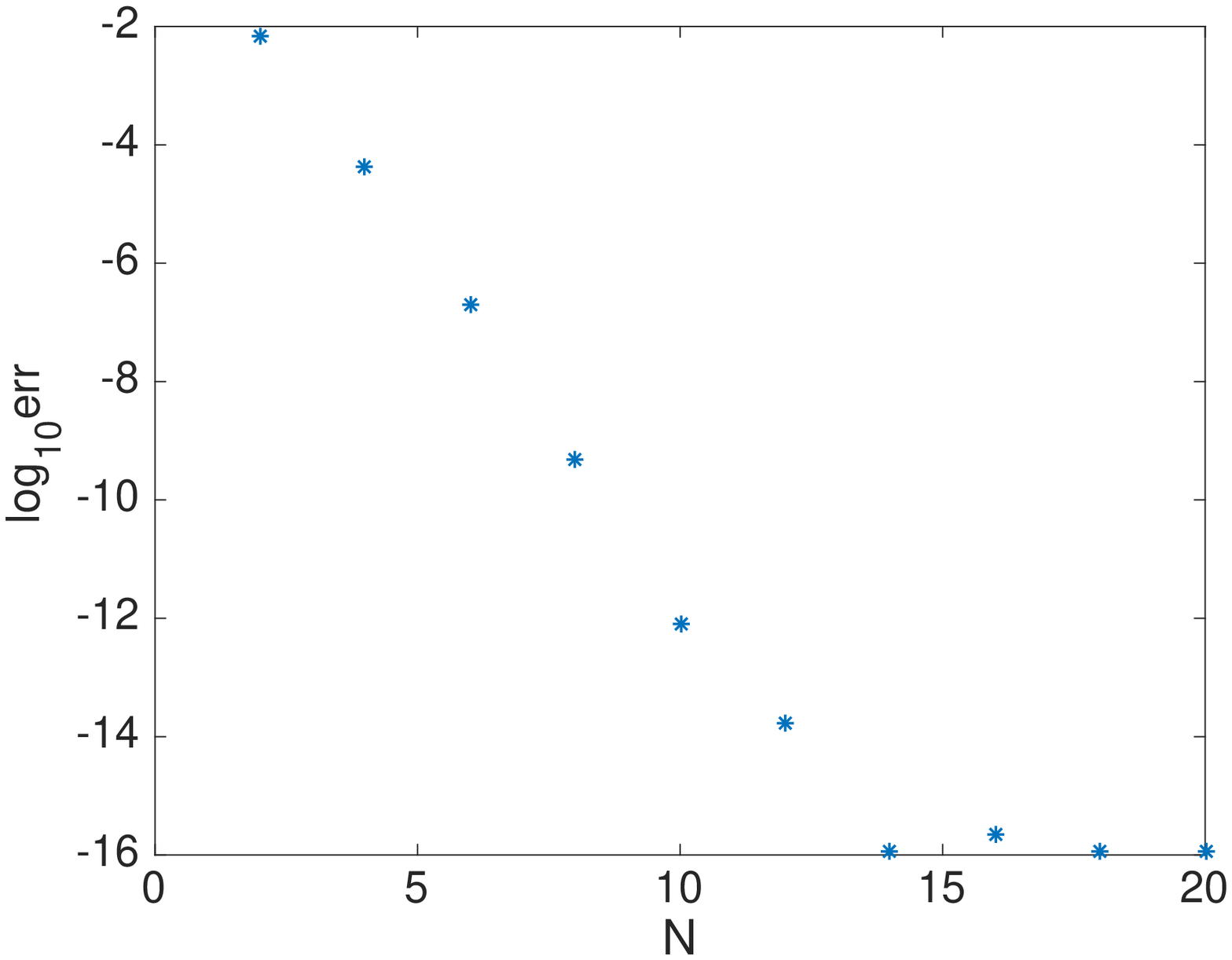}
\includegraphics[width=0.49\textwidth]{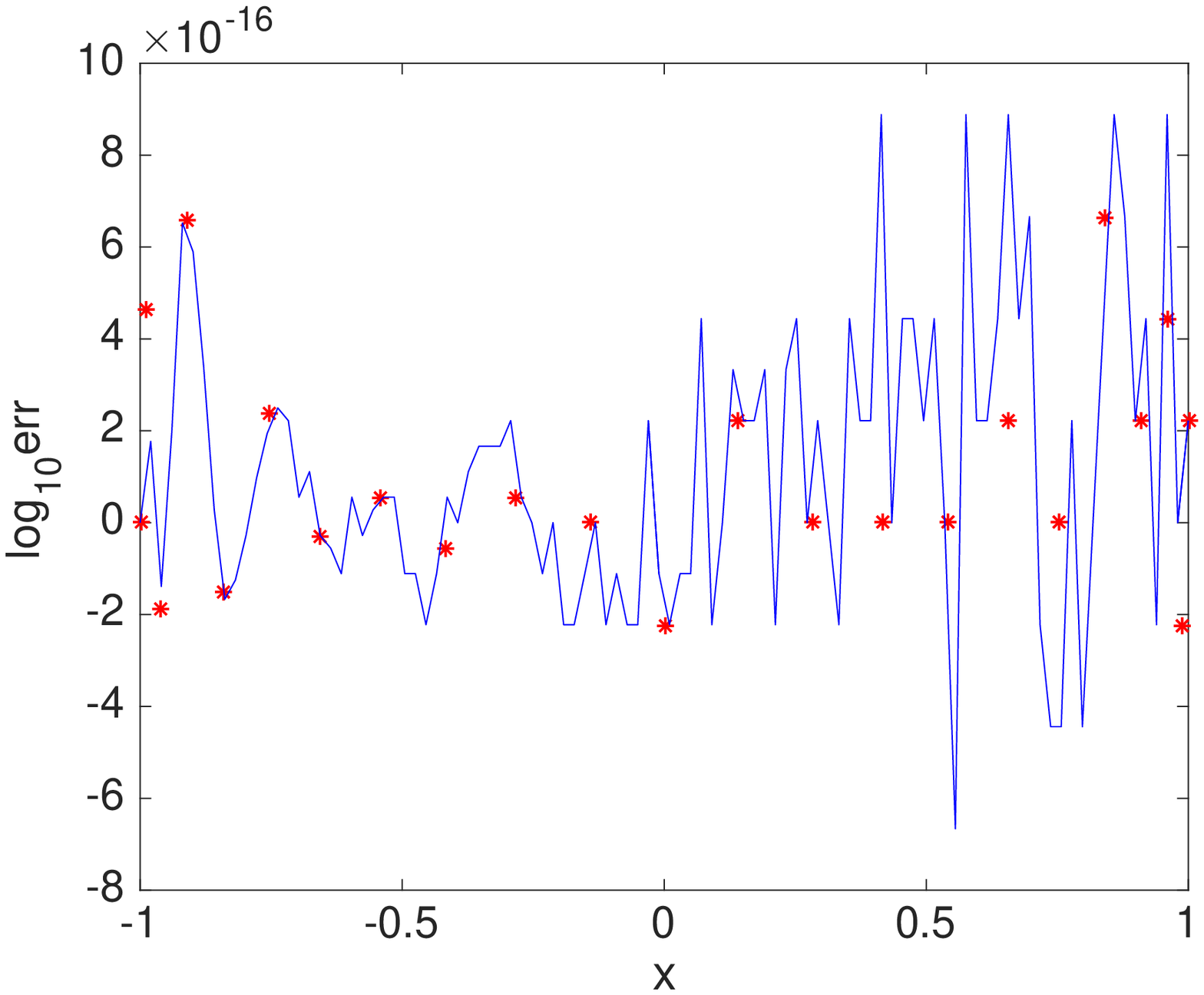}
\end{center}
\vspace{-0.3cm}
\caption{\label{ccfig} Antiderivative of $f=\exp(-x^{2}/2)$ on the 
interval $[0,1]$ 
as approximated by the Clenshaw-Curtis algorithm; 
on the left the difference between the approximation of 
$\int_{0}^{1}f(x)dx$ and $\sqrt{\pi/2}\,\mathrm{erf}(1/\sqrt{2})$ in dependence of $N$; on the right the difference 
between the antiderivative and $\sqrt{\pi/2}\,\mathrm{erf}(x/\sqrt{2})$ for $N=20$ at 
the collocation points (marked with $*$ in red) and for intermediate values 
after barycentric interpolation.}
\end{figure}

The Clenshaw-Curtis algorithm gives in principle only the 
antiderivative of a function at the collocation points. However in the 
present context, the former will be needed on more general values 
of $l$. Since the basis of the approach is a sum of Chebychev 
polynomials, intermediate values can be obtained in principle from formula 
(\ref{approx}). A numerically stable and very efficient way to 
interpolate is to use 
Lagrange interpolation in the barycentric form, see \cite{BT04} and 
references therein. For Chebychev collocation points, the 
interpolation weights can be given explicitly, and a Matlab code for 
this case can be found in \cite{BT04}. The difference between the 
antiderivative of the function $f(x)=\exp(-x^{2}/2)$ and $\sqrt{\pi/2}\,\mathrm{erf}(x/\sqrt{2})$ in 
dependence of $x$ can be seen for $N=20$ in Fig.~\ref{ccfig} on the 
right. The difference on the collocation points is marked with red 
`$*$s'. 
It can be seen that the error introduced by interpolation at 
intermediate points is also smaller than $10^{-15}$. 

As mentioned above, the exponential decay of the numerical error 
with $N$ in the Clenshaw-Curtis algorithm for functions analytic in a 
strip around the real axis in the complex plane is a 
general feature of Chebychev series for such functions. This can be 
seen on the left of Fig.~\ref{ccfigN} where the Chebychev 
coefficients of the antiderivative of $\exp(-x^{2}/2)$ are shown. 
They decrease exponentially, and the numerical 
error due to truncation of the series at $N+1$ terms is actually due 
to the highest order spectral coefficient. Thus the Chebychev 
coefficients also provide an approach to estimate the numerical error 
due to a spectral method by studying the spectral coefficients. For 
the example $a(x)=\exp(-x^{2})$, we just considered in 
Fig.~\ref{ccfig} the term $\int_{0}^{x}\sqrt{a(\tau)}\,d\tau$ in the integral of 
(\ref{2orderWKB}), since we had an independent way to compute the 
exact integral via the error function. This is not the case for the 
terms proportional to $\varepsilon^{2}$ in the integral in 
(\ref{2orderWKB}). But the spectral coefficients of the 
antiderivative of these terms ($(1+x^{2}/2)\exp(x^{2}/2)/4$ for the 
example $a(x)=\exp(-x^{2})$) on the right of Fig.~\ref{ccfigN} 
indicate a similar behavior of the error as in Fig.~\ref{ccfig}: For 
$N\sim20$ the Chebychev coefficients are of the order of the rounding 
error, and further increase of the number of coefficients no longer 
leads to higher accuracy. 
\begin{figure}[htbp]
\begin{center}
\includegraphics[width=0.49\textwidth]{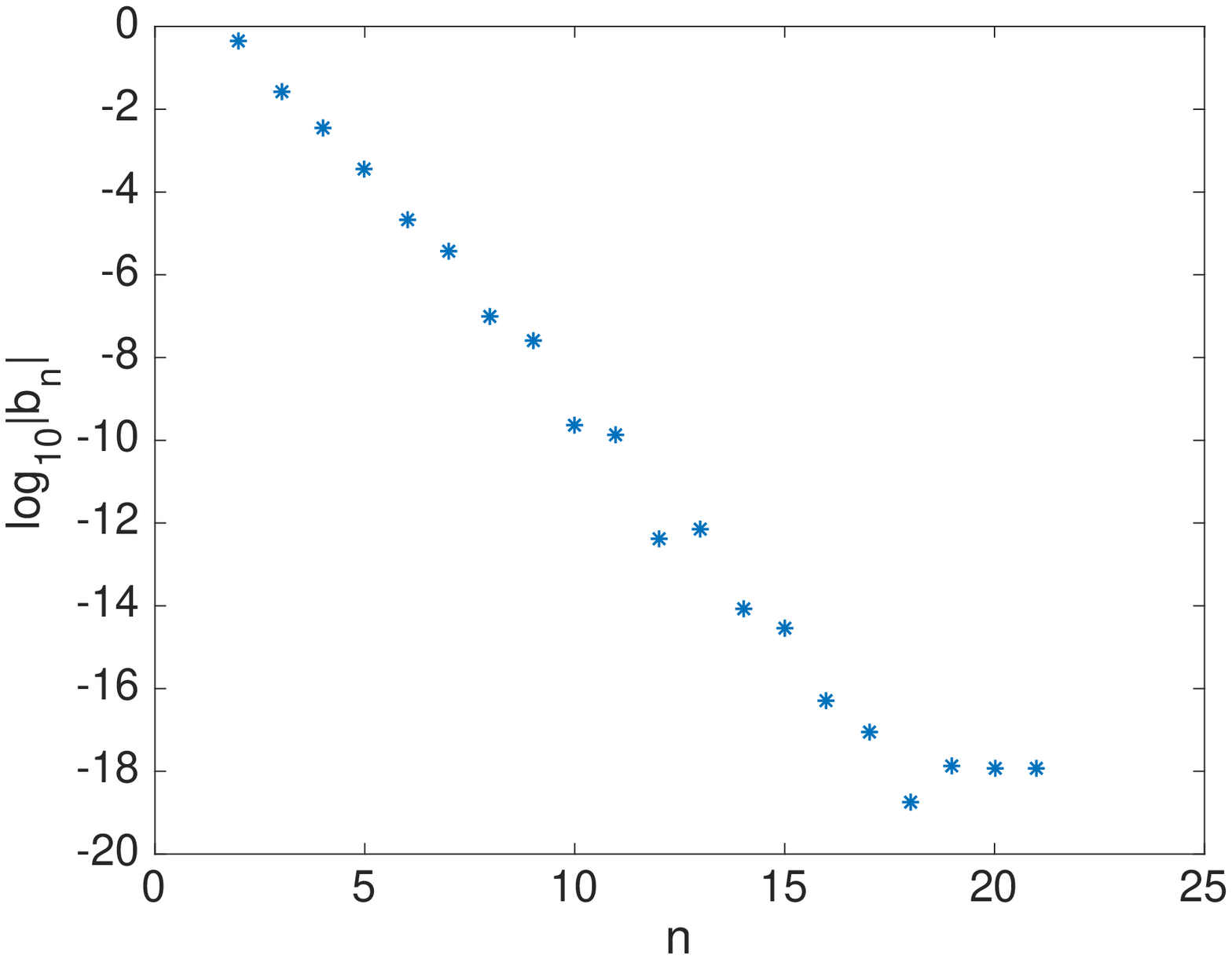}
\includegraphics[width=0.49\textwidth]{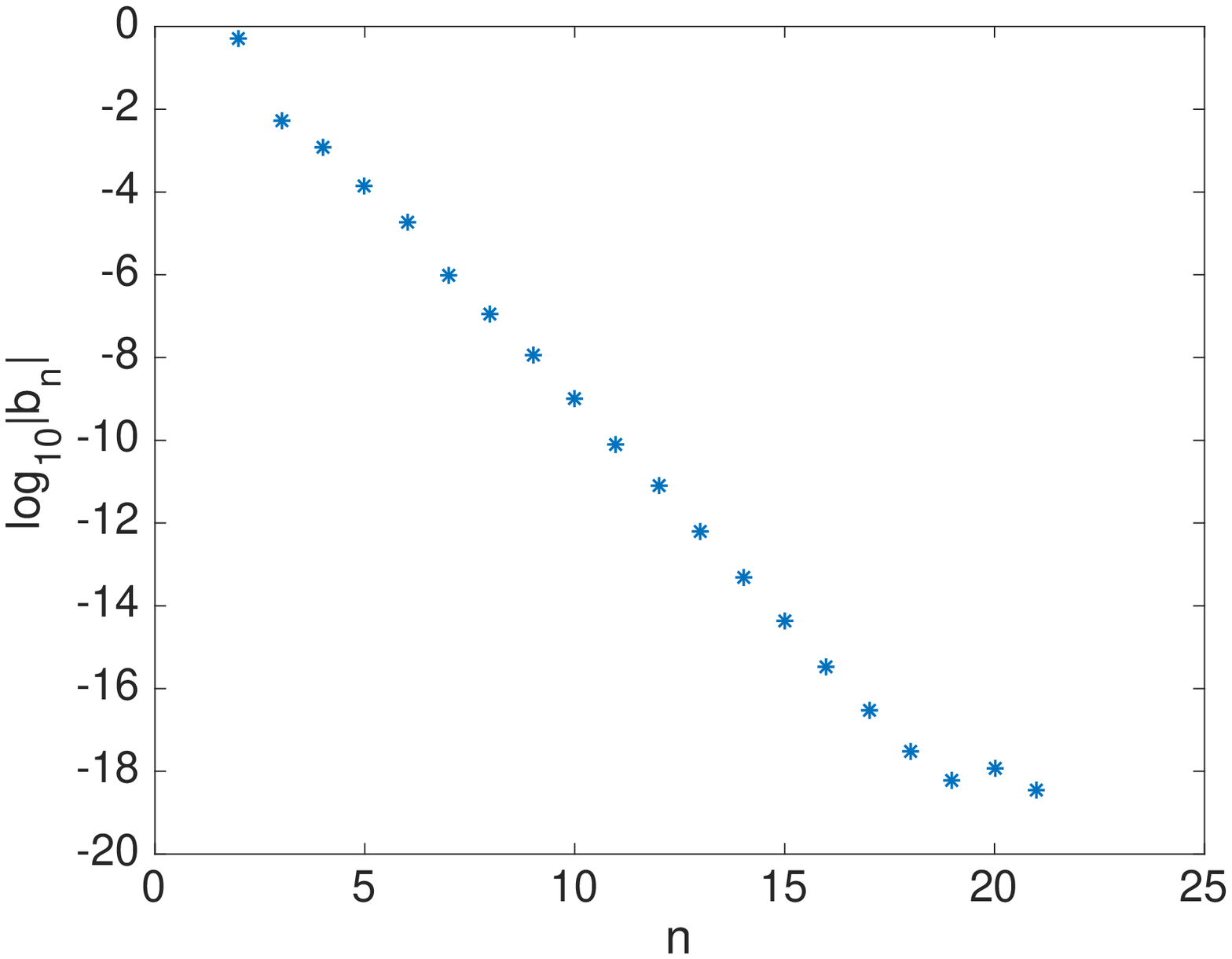}
\end{center}
\vspace{-0.3cm}
\caption{\label{ccfigN} Spectral coefficients for two antiderivatives 
in dependence of the number of collocation points: on the left for 
the function $\exp(-x^{2}/2)$, on the right for 
$(1+x^{2}/2)\exp(x^{2}/2)/4$.}
\end{figure}

\noindent\textbf{Remark:}\\
    The example in Fig.~\ref{ccfig} shows that a spectral approach for 
    $C^{\infty}$ functions allows in practice to reach machine 
    precision with low resolution, here with just 14 Chebychev 
    polynomials. Thus the numerical error reaches a plateau which 
    itself will increase with $N$. The latter is due to the fact that 
    rounding errors pile up with larger values of $N$. With finite 
    difference methods, considerably larger values of $N$ (or 
    equivalently smaller 
    values of $h$) are needed to reach the saturation of the numerical 
    errors, and because of this, the plateau is reached in practice 
    at much higher values than here,  of the order of $10^{-10}$ (see for instance examples 
    in \cite{Tr00}). 

Note that in \cite{ABN11}, the saturation level of the 
    numerical errors was not reached since quadruple precision 
    was used, and since the values of $h$ were not small enough to 
    get there with the used precision. Here, we work in double precision 
    and will reach the saturation level in most cases.

\Section{Numerical results} \label{SEC5}

We shall present now numerical results obtained with the first and second order
WKB-schemes from Section \ref{SEC2}. For our numerical tests we chose
$a(x)=\exp(-x^2)$ on the spatial interval $[0,1]$ with a uniform 
grid, and the initial condition $U_I=(1,\,-i)^\top$. For both schemes 
we shall compare the results obtained with two versions of the 
numerical phase computation: on the one hand by the composite Simpson 
rule (with error order $\gamma=4$) on the WKB-grid $\{x_n\}$, and on 
the other hand by the spectral method from \S\ref{SEC4} along with 
barycentric interpolation at the WKB-grid points $x_n$. Note that we 
use a Chebychev grid with $N=20$ points as in the previous section 
for the computation of the phase $\phi$ in (\ref{2orderWKB}) and 
then interpolate to the equidistant grid for the WKB-scheme. This is 
necessary since Chebychev collocation points are not equidistant. A 
consequence of this approach is that the spectral method computes the 
phase always to machine precision. In fact, the absolute and relative errors of $\tilde\phi_j^{(k)};\,j=1,2;\,k=0,...,4$ are always of the order $10^{-16}-4\cdot10^{-15}$ for this example.

\begin{figure}[htbp]
\begin{center}
\hspace{-13mm}
\includegraphics[width=0.54\textwidth]{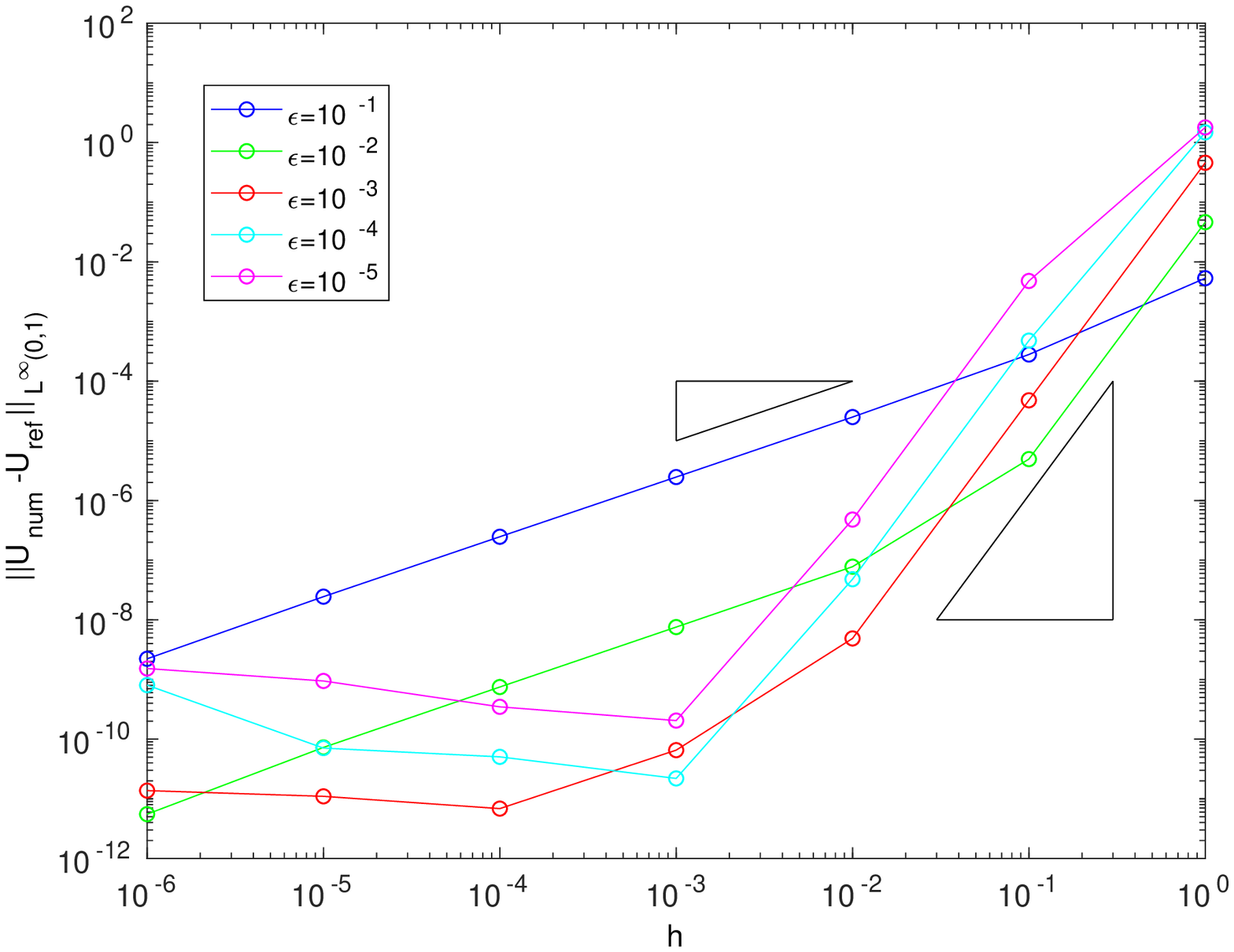}
\includegraphics[width=0.54\textwidth]{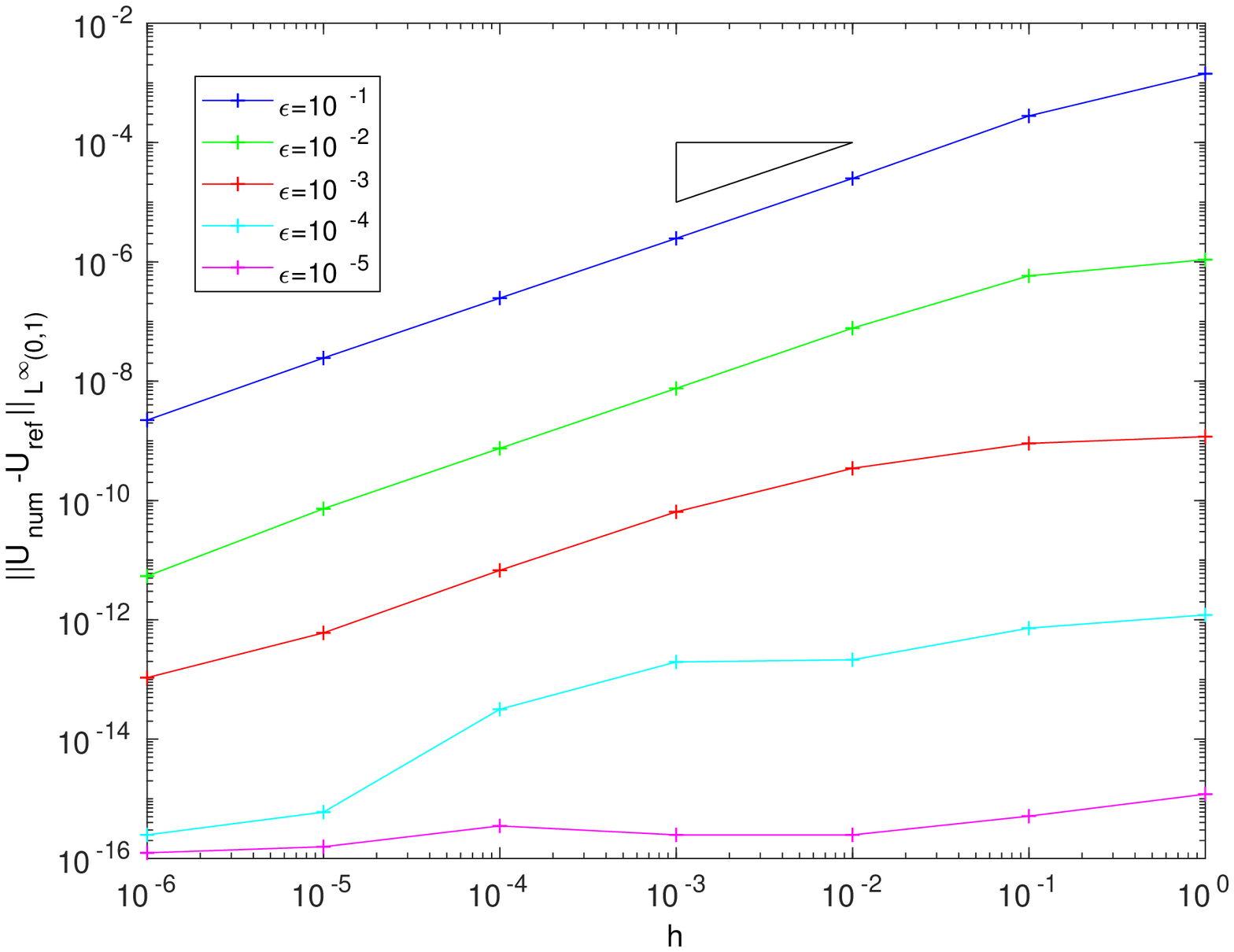}
\end{center}
\vspace{-0.3cm}
\caption{\label{1ORDfig} Error of the first order WKB-method on the interval $[0,1]$ as a function of $h$, for 5 values of $\eps$; on the left the results with the phase $\tilde\phi$ computed via Simpson's rule; on the right the analogous results with $\tilde\phi$ computed via Clenshaw-Curtis algorithm and barycentric interpolation.}
\end{figure}

Fig.\ \ref{1ORDfig} shows the results for the first order method. 
Plotted are the $L^\infty(0,1)$--errors 
of the numerical solution $\{U_n\}$ as a function of $h$ for 5 values of $\eps$. The reference solutions were obtained with the same method, but on a much finer grid. 
For simplicity we shall refer in our discussion to the error estimate \eqref{error_U}. The left plot is obtained with the Simpson rule to compute $\tilde\phi$. For $\eps=10^{-1}$ the second term (i.e.\ the WKB-error) in \eqref{error_U} dominates and the method is clearly first order in $h$, as indicated by the upper slope triangle. 
For $\eps=10^{-2}$ one observes this behavior for $h\le 10^{-2}$.
For smaller values of $\eps$ and large step sizes (e.g.\ $h=1$) the 
error behaves like the first error term in \eqref{error_U}, i.e. 
$h^4/\eps$ due to the Simpson rule, as visualized by the lower slope 
triangle. This leads to an inversion of the 5 error curves at $h=1$ 
($\eps=10^{-5}$ at the top, $\eps=10^{-1}$ at the bottom). However, 
for small step sizes (e.g.\ $h=10^{-6}$) the error term 
$\eps^2\min(\eps,h)$ dominates, such that the inversion of the 5 
error curves w.r.t.\ $\eps$ disappears. But for small values of $h$ 
the error curves are also polluted by round-off errors (due to the 
double precision computations in Matlab). For the phase computation, 
the composite Simpson rule has a worse conditioning than the spectral 
method. Hence the former increases the effect of round-off errors 
here, see the remark in the previous section. Thus 
the error reaches  the saturation level for the Simpson rule at 
higher values of the error than for the 
spectral method. This is also the reason why smaller errors can be 
reached for small $\eps$ and small $h$ in the right figure of 
Fig.~\ref{1ORDfig}. In the left figure, it can be recognized that the 
errors  in the phase computation lead to an effective increase of 
the numerical error with decreasing $h$, and the $\eps$ 
dependence of the related term on (\ref{error_U}) implies that this is 
mainly visible for values of $\eps<10^{-2}$. 

The right plot in Fig.\ \ref{1ORDfig} is obtained with the spectral 
method for $\tilde\phi$, and it reveals that the problematic first term in \eqref{error_U} has been essentially eliminated. As shown by the slope triangle, the method is first order in $h$. For large step sizes (e.g.\ $h=1$) the error behaves like $\calO(\eps^3)$, and for small step sizes (e.g.\ $h=10^{-6}$) roughly like $\calO(\eps^2)$, as predicted by \eqref{error_U}.\\

\begin{figure}[htbp]
\begin{center}
\hspace{-13mm}
\includegraphics[width=0.54\textwidth]{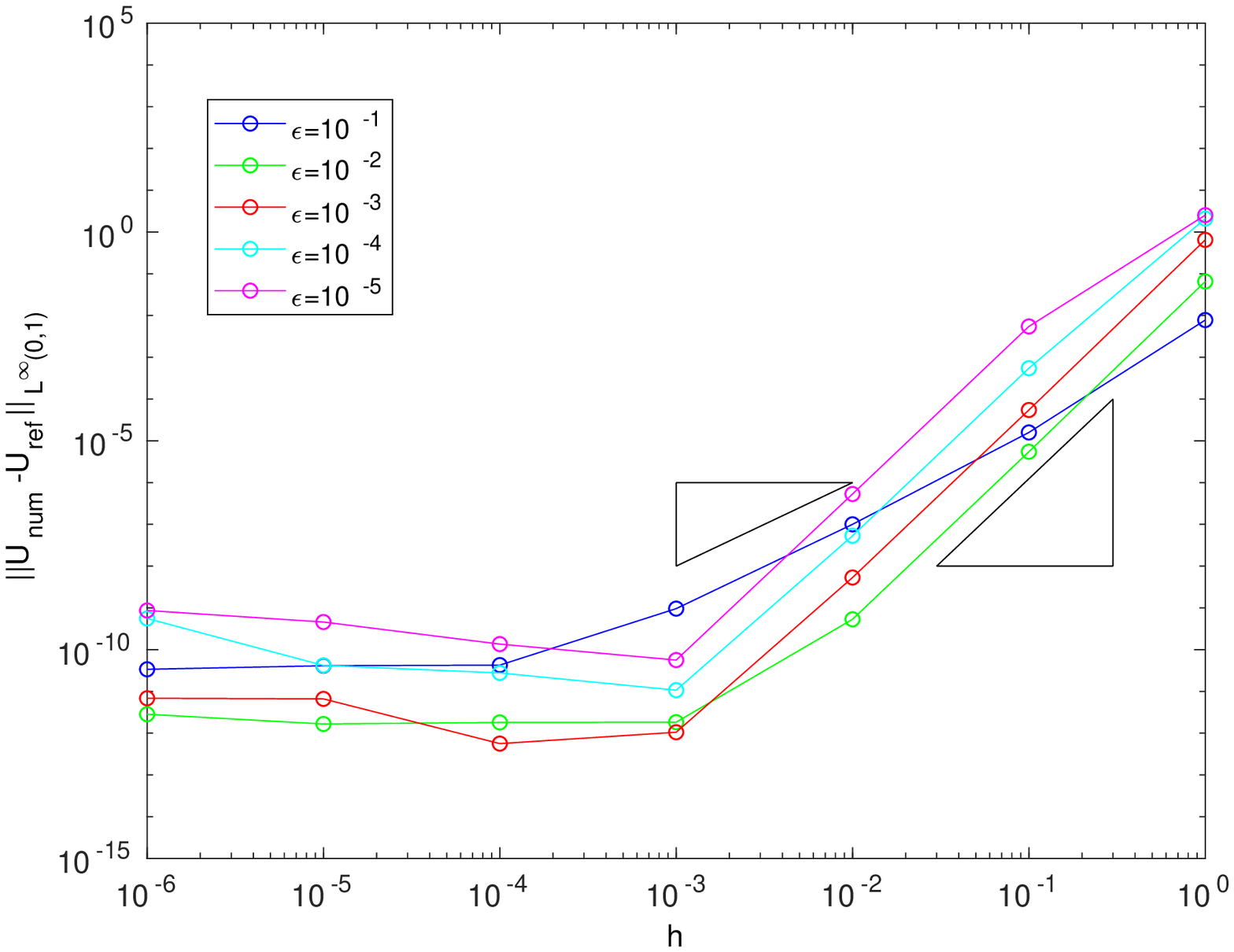}
\includegraphics[width=0.54\textwidth]{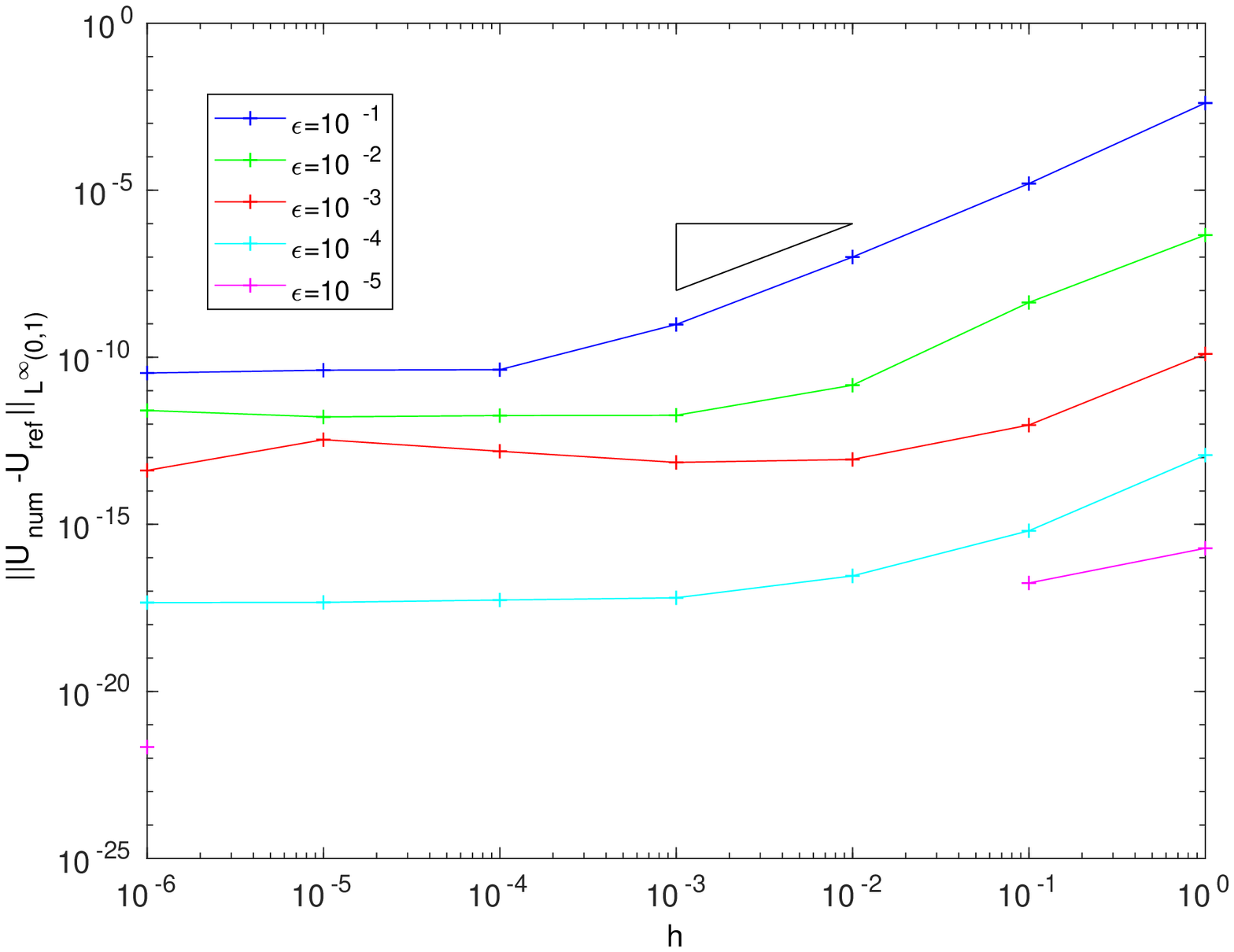}
\end{center}
\vspace{-0.3cm}
\caption{\label{2ORDfig} Error of the second order WKB-method on the interval $[0,1]$ as a function of $h$, for 5 values of $\eps$; on the left the results with the phase $\tilde\phi$ computed via Simpson's rule; on the right the analogous results with $\tilde\phi$ computed via Clenshaw-Curtis algorithm and barycentric interpolation.}
\end{figure}

Fig.\ \ref{2ORDfig} shows the results for the second order method, 
again for 5 values of $\eps$. 
Grosso modo the error behavior is similar 
to the first order method, and we shall compare it to the error 
estimate \eqref{error_U_2ORD}. Of course the errors of the second 
order method are smaller than for the first order method. Therefore, 
the WKB-error reaches the saturation level for small step sizes (usually for 
$h\le 10^{-4}$ -- $10^{-3}$) (note that the WKB-error did not reach the saturation level for the first order method in 
Fig.~\ref{1ORDfig}).

The left plot in Fig.\ \ref{2ORDfig} is obtained with the Simpson 
rule to compute $\tilde\phi$. For $\eps=10^{-1}$ the second term 
(i.e.\ the WKB-error) in \eqref{error_U_2ORD} dominates and the 
method is second order for $h\ge 10^{-3}$, as indicated by the upper 
slope triangle. Again the fact that the maximally achievable accuracy 
with the Simpson method is reached at higher values than with the 
spectral method leads to a slight increase of the error for very 
small values of $h$. Since, in this case, both errors of the WKB method and the 
Simpson integration of the phase are due to rounding errors, there is 
no simple dependence of $\eps$ and $h$ and the errors are of the 
order of $10^{-10}$.

The right plot is obtained with the spectral method for $\tilde\phi$, 
and it reveals that the problematic first term in 
\eqref{error_U_2ORD} has been eliminated again. As shown by the slope 
triangle, the method is second order in $h$ (for $h$ large, i.e.\ 
before the WKB-error reaches the saturation level). Again the better 
conditioning of the spectral method compared to the Simpson method 
allows to achieve smaller numerical errors depending on $\eps$. 
For $\eps=10^{-5}$ the error dropped below the relative machine precision for the values $h=10^{-5}-10^{-2}$. 
Therefore, Matlab's double precision could not compute a positive error value in these cases. Hence, these points are omitted in Fig.\ \ref{2ORDfig}, right.

\begin{figure}[htbp]
\begin{center}
\hspace{-13mm}
\includegraphics[width=0.54\textwidth]{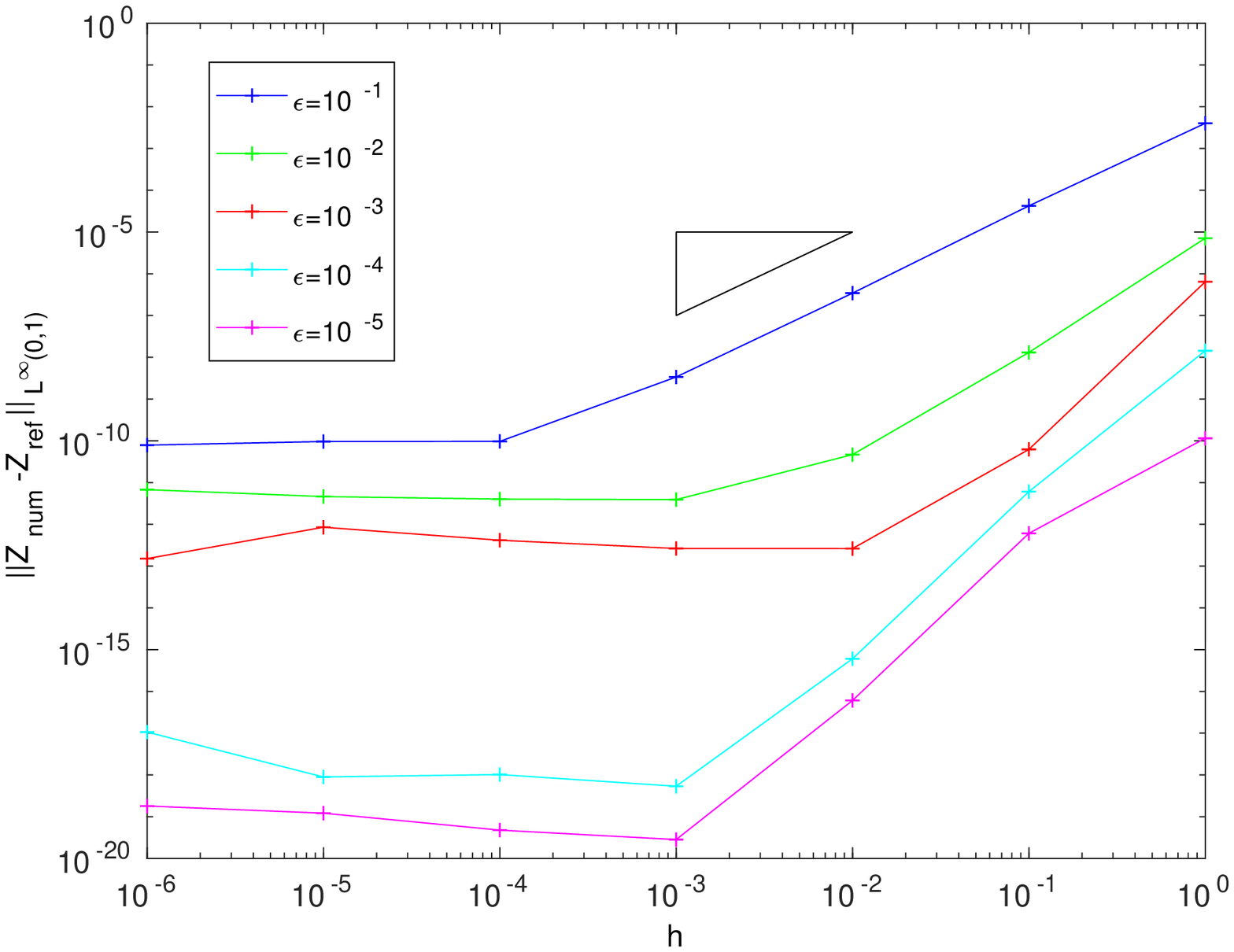}
\includegraphics[width=0.54\textwidth]{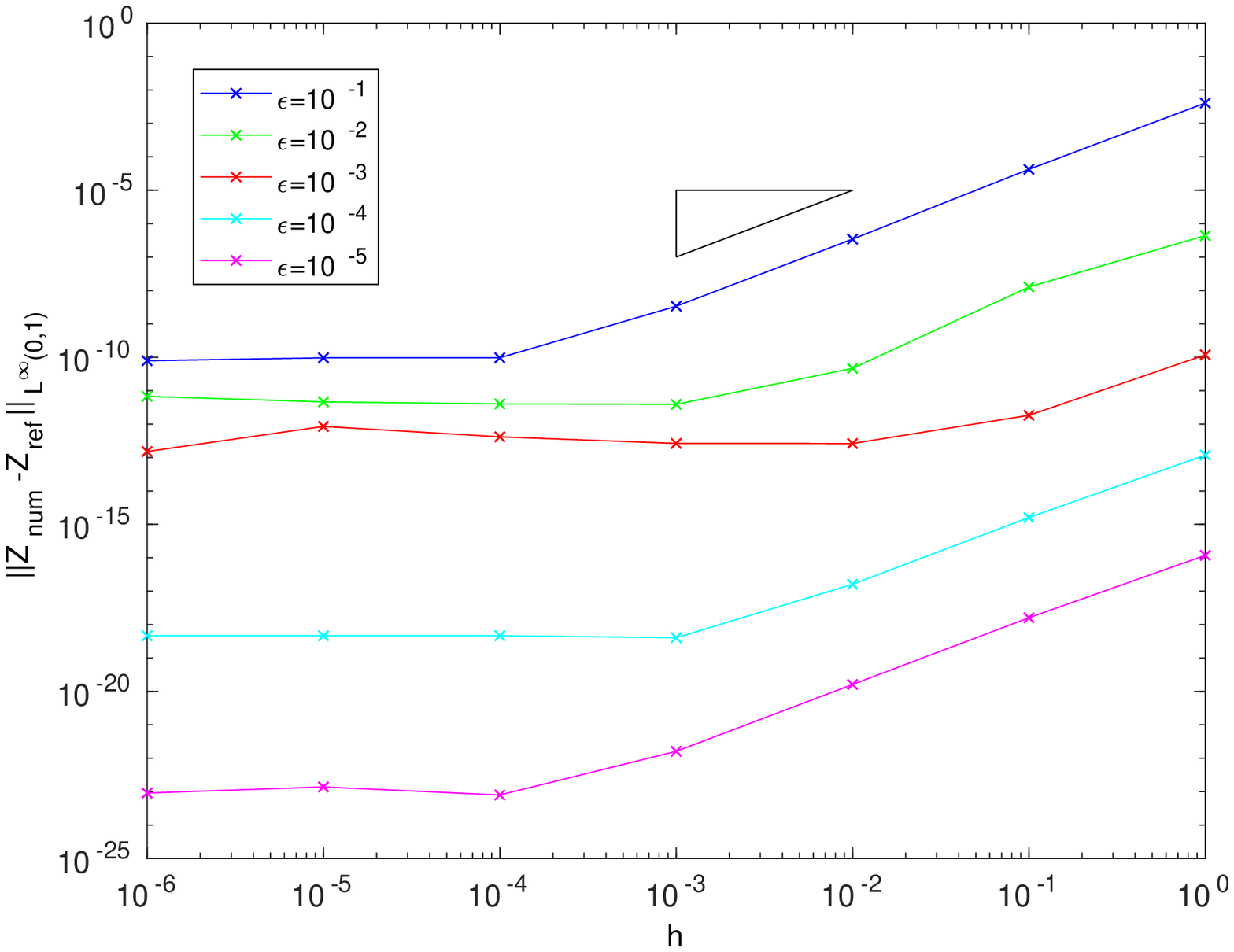}
\end{center}
\vspace{-0.3cm}
\caption{\label{2ORDfigZ} Error of the second order WKB-method on the interval $[0,1]$ as a function of $h$, for 5 values of $\eps$; on the left the results with the phase $\tilde\phi$ computed via Simpson's rule; on the right the analogous results with $\tilde\phi$ computed via Clenshaw-Curtis algorithm and barycentric interpolation.}
\end{figure}

To round off the presentation of the second order method, we present in Fig.\ \ref{2ORDfigZ} the error behavior of the transformed variable $Z$. It satisfies the error bounds from \eqref{Z-error1}, which does \emph{not} include the problematic term $\frac{E}{\eps}$. Note that the right plot includes all $h$-values in the error curve for $\eps=10^{-5}$.

\Section{Conclusion} \label{SEC6}

In this paper we have reviewed the numerical approach of \cite{ABN11} 
to efficiently compute highly oscillatory solutions to the 1D 
Schr\"odinger equation (\ref{EQ_ref}) for small values of the 
semiclassical parameter $\eps$. The method presented in 
\cite{ABN11} uses the leading terms of the WKB 
approximation to solutions of (\ref{EQ_ref}) to reformulate the 
problem in terms of less oscillatory functions. Two methods where given in 
\cite{ABN11}, one of 
first order and one of second order. An error analysis showed that 
the error (\ref{error_U}) has a term due to the numerical computation 
of the phase which is proportional to $1/\eps$ and thus 
problematic in the limit $\eps\to0$. In this paper we have first 
refined the error analysis by taking into account numerical errors in 
the phase computation also in the transformations (\ref{OI}) between the solution 
to the Schr\"odinger equation and the reduced system. In addition the 
phase is now computed with a spectral method showing an exponential 
decrease of the numerical error with the number of modes for 
analytical functions. The advantages of this approach with respect to 
a Simpson method are illustrated for several examples. 

An interesting question as a consequence of this work is whether the spectral approach to the 
numerical computation of the phase can be extended to the complete reduced 
system (\ref{EQZ}), i.e., whether this system can be itself 
efficiently treated with a spectral method. This is not obvious 
since the function $Z$ is known to have oscillations of twice the 
frequency as the solution $\varphi$ of the Schr\"odinger equation, 
but at much smaller amplitude. To check whether spectral methods can 
be efficient in this context will be the subject of further work.

\section*{Acknowledgments}
\quad The first author (AA) was supported by the FWF-doc\-tor\-al school ``Dissipation and dispersion in non-linear partial differential equations'', the bi-national FWF-project I3538-N32, 
and a sponsorship by \emph{Clear Sky Ventures}. This work was 
supported by the  ANR-FWF project ANuI. CK thanks for support by the isite BFC project 
NAANoD, the ANR-17-EURE-0002 EIPHI and by the 
European Union Horizon 2020 research and innovation program under the 
Marie Sklodowska-Curie RISE 2017 grant agreement no. 778010 IPaDEGAN. 

\end{document}